\documentclass[draft,12pt,twoside,intlimits]{amsart}
\usepackage[latin2]{inputenc}
\usepackage{amsmath}
\usepackage{amsthm}
\usepackage{amssymb}
\usepackage[mathscr]{eucal}
\usepackage{enumerate}

\newcommand{\comment}[1]{}

\usepackage{amsfonts}
\usepackage{geometry,amssymb}
\newtheorem{theorem}{\bf Theorem}
\newtheorem{proposition}[theorem]{\bf Proposition}
\newtheorem{corollary}[theorem] {\bf Corollary}
\newtheorem{lemma}{\bf Lemma}

\newtheorem{definition}{\bf Definition}

\newtheorem{remark}{\bf Remark}

\newcommand{\con}{{\rm con~}}

\newcommand{\NN}{\mathbb N}

\newcommand{\RR}{\mathbb R}
\newcommand{\CC}{\mathbb C}

\newcommand{\bfx}{{\bf x}}

\newcommand{\bbb}{{\bf b}}
\newcommand{\bfc}{{\bf c}}
\newcommand{\bfu}{{\bf u}}
\newcommand{\bfv}{{\bf v}}
\newcommand{\bfw}{{\bf w}}
\newcommand{\bfn}{{\bf n}}

\newcommand{\bnu}{{\mathbf \nu}}

\newcommand{\nux}{{\mathbf \nu}{(\bf x})}
\newcommand{\nuy}{{\mathbf \nu}{(\bf y})}

\newcommand{\Floor}[1]{{\left\lfloor{#1}\right\rfloor}}
\newcommand{\Ceil}[1]{{\left\lceil{#1}\right\rceil}}

\def\vp{\varphi}

\def\al{\alpha}

\def\y{{\bf y}}
\def\x{{\bf x}}
\def\z{{\bf z}}
\def\nx{{\bf n(x)}}
\def\ny{{\bf n(y)}}

\def\t{{\bf t}}
\def\w{{\bf w}}
\def\c{{\bf c}}
\def\n{{\bf n}}

\def\tilo{\tilde{\omega}}

\def\intt{\mathop{\rm int}}
\def\<{\langle}
\def\>{\rangle}

\newcounter{oldresult}
\def\theoldresult{\Alph{oldresult}}
\newenvironment{oldresult}{
  \em
  \vskip 0.10in
  \refstepcounter{oldresult}
  \noindent{\bf Theorem\ \theoldresult.}
}{\vskip 0.10in}

\newcounter{oldprop}
\def\theoldprop{\Alph{oldprop}}

\newcounter{oldlemma}
\def\theoldlemma{\Alph{oldlemma}}

\newcounter{oldcor}
\def\theoldcor{\Alph{oldcor}}

\newcounter{oldconjecture}
\def\theoldconjecture{\Alph{oldconjecture}}

\newcounter{hypothesis}
\def\thehypothesis{\Alph{hypothesis}}

\newcounter{rev}
\begin{document}

\title[Blaschke Rolling Ball Theorem]
{A discrete extension of the Blaschke Rolling Ball Theorem}

\author[Sz. Gy. R\'ev\'esz]
{Szil\'ard Gy. R\'ev\'esz}

\address{A. R\'enyi Institute of Mathematics \newline \indent Hungarian
Academy of Sciences, \newline \indent Budapest, P.O.B. 127, 1364
\newline \indent Hungary} \email{revesz@renyi.hu}

%%%%%%%%%%%%%%%%%%%%%%%%%%%%%%%%%%%%%%%%%%%%%%%%%%%%%%%%%%%%%%%%%
%%%%%                                                       %%%%%
%%%%%                     ABSTRACT                          %%%%%
%%%%%                                                       %%%%%
%%%%%%%%%%%%%%%%%%%%%%%%%%%%%%%%%%%%%%%%%%%%%%%%%%%%%%%%%%%%%%%%%

\begin{abstract}
The Rolling Ball Theorem asserts that given a convex body
$K\subset \RR^d$ in Euclidean space and having a $C^2$-smooth
surface $\partial K$ with all principal curvatures not exceeding
$c>0$ at all boundary points, $K$ necessarily has the property
that to each boundary point there exists a ball $B_r$ of radius
$r=1/c$, fully contained in $K$ and touching $\partial K$ at the
given boundary point from the inside of $K$.

In the present work we prove a discrete analogue of the result on
the plane. We consider a certain discrete condition on the
curvature, namely that to any boundary points $\x,\y\in\partial K$
with $|\x-\y|\leq \tau$, the angle $\vp(\bfn_\x,\bfn_\y):= \arccos
\langle \bfn_\x,\bfn_\y \rangle$ of any unit outer normals
$\bfn_\x,\bfn_\y$ at $\x$ and at $\y$, resp., does not exceed a
given angle $\varphi$. Then we construct a corresponding body,
$M(\tau,\varphi)$, which is to lie fully within $K$ while
containing the given boundary point $\x\in\partial K$.

In dimension $d=2$, that is, on the plane, $M$ is almost a regular
$n$-gon, and the result allows to recover the precise form of
Blaschke's Rolling Ball Theorem in the limit.

Similarly, we consider the dual type discrete Blaschke theorems
ensuring certain circumscribed polygons. In the limit, the
discrete theorem enables us to provide a new proof for a strong
result of Strantzen assuming only a.e. existence and lower
estimations on the curvature.

For $d\geq 3$, directly we can derive only a weaker, quasi-precise
form of the discrete inscribed ball theorem, while no space
version of the circumscribed ball theorem is found. However, at
least the higher dimensional smooth cases follow already from the
plane versions of the smooth theorems, which obtain as limiting
cases also from our discrete versions.
\end{abstract}

\thanks{Supported in part by the Hungarian National Foundation for
Scientific Research, Project \#s K-61908 and K-72731.}

\maketitle

%%%%%%%%%%%%%%%%%%%%%%%%%%%%%%%%%%%%%%%%%%%%%%%%%%%%%%%%%%%%%%%%%
%%%%%                                                       %%%%%
%%%%%                     SECTION 1                         %%%%%
%%%%%                                                       %%%%%
%%%%%%%%%%%%%%%%%%%%%%%%%%%%%%%%%%%%%%%%%%%%%%%%%%%%%%%%%%%%%%%%%

\section{Introduction}\label{sec:intro}

Let $\RR^d$ be the usual Euclidean space of dimension $d$,
equipped with the Euclidean distance $|\cdot|$. Our starting point
is the following classical result of Blaschke \cite[p. 116]{Bla}.

\begin{oldresult}{\bf (Blaschke).}\label{oldth:Blaschke} Assume
that the convex domain $K\subset \RR^2$ has $C^2$ boundary
$\Gamma=\partial K$ and that with the positive constant
$\kappa_0>0$ the curvature satisfies $\kappa(\z)\le \kappa_0$ at
all boundary points $\z\in\Gamma$. Then to each boundary points
$\z\in\Gamma$ there exists a disk $D_R$ of radius $R=1/\kappa_0$,
such that $\z\in\partial D_R$, and $D_R \subset K$.
\end{oldresult}

Note that the result, although seemingly local, does not allow for
extensions to non-convex curves $\Gamma$. One can draw pictures of
leg-bone like shapes of arbitrarily small upper bound of
(positive) curvature, while at some points of touching containing
arbitrarily small disks only. The reason is that the curve, after
starting off from a certain boundary point $\x$, and then leaning
back a bit, can eventually return arbitrarily close to the point
from where it started: hence a prescribed size of disk cannot be
inscribed. %%%% See Picture 1.

On the other hand the Blaschke Theorem extends to any dimension
$d\in\NN$. Also, the result has a similar, dual version, too, see
\cite[p. 116]{Bla}.

\begin{oldresult}{\bf (Blaschke.)}\label{thold:Blaschkeout} Assume
that $K\subset \RR^2$ is a convex domain with $C^2$-smooth
boundary curve $\gamma$ having curvature $\kappa\geq \kappa_0$ all
over $\gamma$. Then to all boundary point $\x\in\gamma$ there
exists a disk $D_R$ of radius $R=1/\kappa_0$, such that
$\x\in\partial D_R$, and $K\subset D_R$.
\end{oldresult}

In Section \ref{sec:preliminaries} we introduce a few notions and
recall auxiliary facts. In \S \ref{sec:discreteresults} we
formulate and prove the two basic results -- the discrete forms of
the Blaschke Theorems -- of our paper. Then we show how our
discrete approach yields a new, straightforward proof for a more
involved sharpening of Theorem \ref{oldth:Blaschke}, originally
due to Strantzen.

This all concerns dimension 2. Only in Section \ref{s:highdim}
will we consider the case of higher dimensional Euclidean spaces.
Certain corresponding results hold also in $\RR^d$, but they are
less satisfactory, as the classical Blaschke theorem cannot be
recovered from them in the limit. Nevertheless, it is worthy to
formulate them, in view of certain applications in multivariate
approximation, what we have in mind when analyzing these
questions.

\section{Preliminaries, geometrical notions}\label{sec:preliminaries}

Recall that the term planar {\em convex body} stands for a
compact, convex subset of $\CC\cong\RR^2$ having nonempty
interior. For a (planar) convex body $K$ any interior point $z$
defines a parametrization $\gamma(\varphi)$ -- the usual polar
coordinate representation of the boundary $\partial K$, -- taking
the unique point \hbox{$\{z+te^{i\varphi}:\,t\in (0,\infty)\} \cap
\partial K$} for the definition of $\gamma(\varphi)$. This defines
the closed Jordan curve $\Gamma=\partial K$ and its
parametrization $\gamma : [0,2\pi] \to \CC$. By convexity, from
any boundary point $\zeta=\gamma(\theta)\in
\partial K$, locally the chords to boundary points with parameter
$<\theta$ or with $>\theta$ have arguments below and above the
argument of the direction of any supporting line at $\zeta$. Thus
the tangent direction or argument function $\alpha_{-}(\theta)$
can be defined as e.g. the supremum of arguments of chords from
the left; similarly, $\alpha_{+}(\theta):=\inf \{\arg (z-\zeta)~:~
z=\gamma(\varphi),~ \varphi>\theta \}$, and any line
$\zeta+e^{i\beta}\RR$ with $\alpha_{-}(\theta)\le \beta \le
\alpha_{+}(\theta)$ is a supporting line to $K$ at
$\zeta=\gamma(\theta)\in\partial K$. In particular the curve
$\gamma$ is differentiable at $\zeta=\gamma(\theta)$ if and only
if $\alpha_{-}(\theta)=\alpha_{+}(\theta)$; in this case the
tangent of $\gamma$ at $\zeta$ is $\zeta+e^{i\alpha}\RR$ with the
unique value of $\alpha=\alpha_{-}(\theta)=\alpha_{+}(\theta)$. It
is clear that interpreting $\alpha_{\pm}$ as functions on the
boundary points $\zeta\in\partial K$, we obtain a
parametrization-independent function. In other words, we are
allowed to change parameterizations to arc length, say, when in
case of $|\Gamma|=\ell$ ($|\Gamma|$ meaning the length of $\Gamma
:=\partial K$) the functions $\alpha_{\pm}$ map $[0,\ell]$ to
$[0,2\pi]$.

Observe that $\alpha_{\pm}$ are nondecreasing functions with total
variation ${\rm Var}\,[\alpha_{\pm}] = 2\pi$, and that they have a
common value precisely at continuity points, which occur exactly
at points where the supporting line to $K$ is unique. At points of
discontinuity $\alpha_{\pm}$ is the left-, resp. right continuous
extension of the same function. For convenience, and for better
matching with \cite{BS}, we may even define the function
$\alpha:=(\alpha_{+}+\alpha_{-})/2$ all over the parameter
interval.

For obvious geometric reasons we call the jump function
$\beta:=\alpha_{+}-\alpha_{-}$ the {\em supplementary angle}
function. In fact, $\beta$ and the usual Lebesgue decomposition of
the nondecreasing function $\alpha_{+}$ to
$\alpha_{+}=\sigma+\alpha_{*}+\alpha_{0}$, consisting of the pure
jump function $\sigma$, the nondecreasing singular component
$\alpha_{*}$, and the absolute continuous part $\alpha_0$, are
closely related. By monotonicity there are at most countable many
points where $\beta(x)>0$, and in view of bounded variation we
even have $\sum_x \beta(x) \le 2\pi$, hence the definition
$\mu:=\sum_x \beta(x)\delta_x$ defines a bounded, non-negative
Borel measure on $[0,2\pi)$. Now it is clear that
$\sigma(x)=\mu([0,x])$, while $\alpha_{*}'=0$ a.e., and $\alpha_0$
is absolutely continuous. In particular, $\alpha$ or $\alpha_{+}$
is differentiable at $x$ provided that $\beta(x)=0$ and $x$ is not
in the exceptional set of non-differentiable points with respect
to $\alpha_{*}$ or $\alpha_0$. That is, we have differentiability
almost everywhere, and
\begin{align}\label{aedifferentiability}
\int_x^y \alpha' =& \alpha_0(y)-\alpha_0(x) = \lim_{z\to x-0}
\alpha_0(y)-\alpha_0(z) \notag
\\=&
\lim_{z\to x-0} \left\{
[\alpha_{+}(y)-\sigma(y)-\alpha_{*}(y))]-[\alpha_{+}(z)-\sigma(z)-\alpha_{*}(z)]\right\}
\notag \\ = & \alpha_{+}(y)- \beta(y)
%%%% \sigma(\{y\})
- \mu([x,y)) - \lim_{z\to x-0} \alpha_{+}(z) - \lim_{z\to x-0}
[\alpha_{*}(y)-\alpha_{*}(z)] \le \alpha_{-}(y)-\alpha_{+}(x)~.
\end{align}
It follows that
\begin{equation}\label{differentiallarge}
\alpha'(t)\ge \lambda \qquad \text{a.e}. \quad t\in [0,a]
\end{equation}
holds true if and only if we have
\begin{equation}\label{fixchange}
\alpha_{\pm}(y)-\alpha_{\pm}(x) \ge \lambda (y-x) \qquad \forall
x,y \in [0,a]~.
\end{equation}
Here we restricted ourselves to the arc length parametrization
taken in positive orientation. Recall that one of the most
important geometric quantities, curvature, is just
$\kappa(s):=\alpha'(s)$, whenever parametrization is by arc length
$s$.

Thus we can rewrite \eqref{differentiallarge} as
\begin{equation}\label{curvaturesmall}
\kappa(t) \ge \lambda \qquad \text{a.e}. \quad t\in [0,a]~,
\end{equation}
or, with radius of curvature $\rho(t):=1/\kappa(t)$ introduced
(writing $1/0=\infty$),
\begin{equation}\label{curvradlarge}
\rho(t) \le \frac{1}{\lambda} \qquad \text{a.e}.\quad t\in [0,a]~.
\end{equation}
Again, $\rho$ is a parametrization-invariant quantity (describing
the radius of the osculating circle). Actually, it is easy to
translate all these conditions to arbitrary parametrization of the
tangent angle function $\alpha$. Since also curvature and radius
of curvature are parametrization-invariant quantities, all the
above hold for any parametrization.

Moreover, with a general parametrization let
$|\Gamma(\eta,\zeta)|$ stand for the length of the
counterclockwise arc $\Gamma(\eta,\zeta)$ of the rectifiable
Jordan curve $\Gamma$ between the two points $\zeta, \eta \in
\Gamma=\partial K$. We can then say that the curve satisfies a
Lipschitz-type increase or {\em subdifferential condition}
whenever
\begin{equation}\label{subdiffcond}
|\alpha_{\pm}(\eta)-\alpha_{\pm}(\zeta)| \ge \lambda
|\Gamma(\eta,\zeta)| \qquad (\forall \zeta, \eta \in \Gamma)~,
\end{equation}
here meaning by $\alpha_{\pm}(\xi)$, for $\xi\in\Gamma$, not
values in $[0,2\pi)$, but a locally monotonously increasing branch
of $\alpha_{\pm}$, with jumps in $(0,\pi)$, along the
counterclockwise arc $\Gamma(\eta,\zeta)$ of $\Gamma$. Clearly,
the above considerations show that all the above are equivalent.

In the paper we use the notation $\alpha$ (and also
$\alpha_{\pm}$) for the tangent angle, $\kappa$ for the curvature,
and $\rho$ for the radius of curvature. The counterclockwise taken
right hand side tangent unit vector(s) will be denoted by $\t$,
and the outer unit normal vectors by $\bfn$. These notations we
will use basically in function of the arc length parametrization
$s$, but with a slight abuse of notation also
$\alpha_{-}(\varphi)$, $\t(\x)$, $\nx$ etc. may occur with the
obvious meaning.

Note that $\t(\x)=i\nx)$ and also $\t(\x)=\dot{\gamma}(s)$ when
$\x=\x(s)\in\gamma$ and the parametrization/differentiation,
symbolized by the dot, is with respect to arc length; moreover,
with $\nu(s):\arg(\bfn(\x(s))$ we obviously have $\alpha\equiv\nu
+\pi/2 \mod 2\pi$ at least at points of continuity of $\alpha$ and
$\nu$. To avoid mod $2\pi$ equality, we can shift to the universal
covering spaces and maps and consider $\widetilde{\alpha},
\widetilde{\nu}$, i.e. $\widetilde{\t}, \widetilde{\bfn}$ -- e.g.
in case of $\widetilde{\bfn}$ we will somewhat detail this right
below. However, note a slight difference in handling $\alpha$ and
$\widetilde{\bfn}$: the first is taken as a singlevalued function,
with values $\alpha(s):=\frac12\{\alpha_{-}(s)+\alpha_{+}(s)\}$ at
points of discontinuity, while $\widetilde{\bfn}$ is a multivalued
function attaining a full closed interval
$[\widetilde{\bfn}_{-}(s),\widetilde{\bfn}_{+}(s)]$ whenever $s$
is a point of discontinuity. Also recall that curvature, whenever
it exists, is
$|\ddot{\gamma}(s)|=\alpha'(s)=\widetilde{\bfn}'(s)$.

In this work we mean by a multi-valued function $\Phi$ from $X$ to
$Y$ a (non-empty-valued) mapping $\Phi:X\to
2^Y\setminus\{\emptyset\}$, i.e. we assume that the domain of
$\Phi$ is always the whole of $X$ and that $\emptyset \ne
\Phi(x)\subset Y$ for all $x\in X$. Recall the notions of modulus
of continuity and minimal oscillation in the full generality of
multi-valued functions between metric spaces.

\begin{definition}[\bf modulus of continuity and minimal oscillation]
\label{modcontdef} Let $(X,d_X)$ and $(Y,d_Y)$ be metric spaces.
We call the \emph{modulus of continuity} of the multivalued
function $\Phi$ from $X$ to $Y$ the quantity
\begin{equation}\label{modcontmetr}
\omega(\Phi,\tau):=\sup \{ d_Y(y,y') ~:~ x,x'\in X,~d_X(x,x') \leq
\tau,~y\in \Phi(x),~y'\in\Phi(x') \}.\notag
\end{equation}
Similarly, we call \emph{minimal oscillation} of $\Phi$ the
quantity
\begin{equation}\label{minoscmetr}
\Omega(\Phi,\tau):= \inf\{ d_Y(y,y') ~:~ x,x'\in X,~d_X(x,x') \geq
\tau,~y\in \Phi(x),~y'\in\Phi(x') \}.\notag
\end{equation}
\end{definition}

If we are given a multi-valued \emph{unit vector function}
$\bfv(\x) :H\to 2^{S^{d-1}}\setminus \{\emptyset\}$, where
$H\subset \RR^d$ and $S^{d-1}$ is the unit ball of $\RR^d$, then
the derived formulae become:
\begin{equation}\label{modcontuv}
\omega(\tau):=\omega(\bfv,\tau):=\sup \{ \arccos \langle \bfu,
\bfw \rangle \,:~\,\x,\y\in H,~|\x-\y| \leq \tau,~ \bfu\in
\bfv(\x),\bfw\in\bfv(\y) \},
\end{equation}
and
\begin{equation}\label{minoscuv}
\Omega(\tau):=\Omega(\bfv,\tau):=\inf \{ \arccos \langle \bfu,\bfw
\rangle ~:~ \x,\y\in H,~|\x-\y| \geq \tau ,~ \bfu\in
\bfv(\x),~\bfw\in\bfv(\y)\}.
\end{equation}

For a \emph{planar} multi-valued unit vector function $\bfv: H\to
2^{S^1}\setminus \{\emptyset\}$, where $H\subset \RR^2\simeq \CC$
and $S^1$ is the unit circle in $\RR^2$, we can parameterize the
unit circle $S^1$ by the corresponding angle $\varphi$ and thus
write $\bfv(\x)=e^{i\Phi(\x)}$ with $\Phi(\x):=\arg( \bfv(\x))$
being the corresponding angle. We will somewhat elaborate on this
observation in the case when our multi-valued vector function is
the outward normal vector(s) function $\nx$ of a closed convex
curve.

Let $\gamma $ be the boundary curve of a convex body in ${\Bbb
R}^2$, which will be considered as oriented counterclockwise, and
let the multivalued function $\bold{n(x)}: \gamma \to 2^{S^1}
\setminus \{ \emptyset \}$ be defined as the set of all outward
unit normal vectors of $\gamma $ at the point $\bold{x} \in \gamma
$. Observe that the set $\bfn ({\bold{x}})$ of the set of values
of $\bfn $ at any ${\bold{x}} \in \gamma $ is either a point, or a
closed segment of length less than $\pi $. Then there exists a
unique lifting $\tilde \bfn $ of $\bfn $ from the universal
covering space $\tilde {\gamma } (\simeq \RR$, see below) of
$\gamma $ to the universal covering space ${\Bbb R}=\tilde {S^1}$
of $S^1$, with the respective universal covering maps $\pi
_{\gamma }: \tilde {\gamma } \to \gamma $ and $\pi _{S^1}: \tilde
{S^1} \to S^1$, with properties to be described below. Here we do
not want to recall the concept of the universal covering spaces
from algebraic topology in its generality, but restrict ourselves
to give it in the situation described above. As already said,
$\tilde {S^1}={\Bbb R}$ and the corresponding universal covering
map is $\pi _{S^1}: x \to (\cos x, \sin x )$ (We consider, as
usual, $S^1$ as ${\Bbb R} \mod 2 \pi $.) Similarly, for $\gamma $
we have $\tilde \gamma ={\Bbb R}$, with universal covering map
$\pi _{\gamma }: {\Bbb R} \to \gamma $ given in the following way.
Let us fix some arbitrary point $\bold{x_0} \in \gamma $, (the
following considerations will be independent of $\bold{x_0}$, in
the natural sense). Let us denote by $\ell$ the length of $\gamma
$. Then for $\lambda \in {\Bbb R}= \tilde {\gamma }$ we have that
$\pi _{\gamma } (\lambda ) \in \gamma $ is that unique point
$\bold{x}$ of $\gamma $, for which the counterclockwise measured
arc $\bold{x_0x}$ has a length $\lambda \mod \ell$.

Now we describe the postulates for the multivalued function
$\tilde {\bfn }: {\Bbb R}=\tilde {\gamma } \to \tilde {S^1}={\Bbb
R}$, which determine it uniquely. First of all, we must have the
equality $\pi _{S^1} \circ \tilde {\bfn }=\bfn \circ \pi _{\gamma
}$, where $\circ $ denotes the composition of two multivalued
functions. (In algebraic topology this is called
{\it{commutativity of a certain square of mappings}}.) Second, the
values of $\tilde \bfn $ must be either points or non-degenerate
closed intervals (of length less than $\pi $; however this last
property follows from the other ones). Third, $\tilde \bfn $ must
be non-decreasing in the following sense: for $\lambda _1, \lambda
_2 \in {\Bbb R},\,\, \lambda _1 < \lambda _2$ we have $r_1 \in
\tilde{\bfn} (\lambda _1), r_2 \in \tilde{\bfn}(\lambda_2)
\Longrightarrow r_1 \le r_2$. Further, $\tilde{\bfn} $ must be a
non-decreasing multivalued function, continuous from the left,
i.e., for any $\lambda \in {\Bbb R}$ we have that for any
$\varepsilon >0$ there exists a $\delta >0$, such that $\cup_{\mu
\in (\lambda - \delta , \lambda )} \tilde{\bfn} (\mu ) \subset
(\min \tilde{\bfn} (\lambda )-\varepsilon , \min \tilde{\bfn}
(\lambda ))$. Analogously, $\tilde{\bfn} $ must be a
non-decreasing multi-valued function continuous from the right,
i.e., for any $\lambda \in {\Bbb R}$ we have that for any
$\varepsilon
>0$ there exists a $\delta >0$, such that $\cup_{\mu \in (\lambda
, \lambda + \delta )} \tilde{\bfn} (\mu ) \subset (\max
\tilde{\bfn} (\lambda ), \max \tilde{\bfn} (\lambda )+\varepsilon
)$. These are all the postulates for the multi-valued function
$\tilde{\bfn}$. It is clear, that $\tilde {\bfn }$ exists and is
uniquely determined, for fixed ${\bold{x_0}}$ (and, for
$\bold{x_0}$ arbitrary, only the parametrization of ${\Bbb
R}=\tilde {\gamma }$ changes, by a translation.)

The above listed properties imply still one important property of
the multi-valued function $\tilde {\bfn }$: we have for any
$\lambda \in {\Bbb R}$ that $\tilde{\bfn} (\lambda
+\ell)=\tilde{\bfn} (\lambda )+2\pi $.

\begin{definition} We define the modulus of continuity of the
multi-valued normal vector function $\bfn(\x)$ \emph{with respect
to arc length} as the (ordinary) modulus of continuity of the
multi-valued lift-up function $\tilde {\bfn } : {{\Bbb{R}}} \to
{\Bbb R}\setminus\{\emptyset\}$, i.e. as
\begin{align}\label{omegaiv}
\tilde{\omega}(\tau):=\tilde {\omega} (\bfn, \tau )& :=\omega
(\tilde {\bfn }, \tau )\notag \\ &:= \sup \{ |r_1-r_2| \mid r_1
\in \tilde {\bfn } (\lambda _1),r_2 \in \tilde{\bfn} (\lambda
_2),\,\,\lambda _1, \lambda _2\in {\Bbb R}, | \lambda _1 - \lambda
_2 | \le \tau \} .
\end{align}
Similarly, we define the minimal oscillation of the multi-valued
normal vector function $\bfn(\x)$ \emph{with respect to arc
length} as the (ordinary) minimal oscillation function of $\tilde
{\bfn }$, i.e. as
\begin{align}\label{Omegaiv}
\tilde{\Omega}(\tau):=\tilde {\Omega} (\bfn , \tau )& := \Omega
(\tilde {\bfn }, \tau )\notag \\&:= \inf \{ |r_1-r_2| \mid r_1 \in
\tilde {\bfn } (\lambda _1),r_2 \in \tilde{\bfn} (\lambda
_2),\,\,\lambda _1, \lambda _2 \in {\Bbb R}, | \lambda _1 -
\lambda _2 | \geq \tau \} .
\end{align}
\end{definition}

By writing "modulus of continuity" we do not mean to say anything
like continuity of $\tilde {\bfn }$. In fact, if for some $\lambda
\in {\Bbb R}$ $\tilde {\bfn }(\lambda )$ is a non-degenerate
closed segment, then the left-hand side and right-hand side limits
of $\tilde {\bfn }$ at $\lambda $ - in the sense of the definition
of continuity from the left or right, respectively - are surely
different.

We evidently have that the modulus of continuity of $\tilde{\bfn}$
is subadditive, meaning $\tilde{\omega} (\tau _1 +\tau _2) \le
\tilde{\omega} (\tau _1)+\tilde{\omega} (\tau _2)$, and similarly,
that the minimal oscillation of $\tilde {\bfn }$ is superadditive,
meaning $\tilde{\Omega} (\tau _1 +\tau _2) \ge \tilde{\Omega}
(\tau _1)+\tilde{\Omega} (\tau _2)$. In fact, a standard property
of the modulus of continuity of \emph{any (non-empty valued)
multivalued function from $\RR$ (or from any convex set, in the
sense of metric intervals) to $\RR$} is subadditivity, and
similarly, minimal oscillation of such a function is
superadditive. These properties with non-negativity and
non-decreasing property also imply that
$\tilde{\omega}(\tau)/\tau$ and $\tilde{\Omega}(\tau)/\tau$ have
limits when $\tau\to 0$; moreover, $\lim_{\tau\to 0}
\tilde{\omega}(\tau)/\tau =\sup \tilde{\omega}(\tau)/\tau$ and
$\lim_{\tau\to 0} \tilde{\Omega}(\tau)/\tau =\inf
\tilde{\Omega}(\tau)/\tau$. Note that metric convexity is
essential here, so e.g. it is not clear if in $\RR^d$ any proper
analogy could be established.

%%%%%%%%%%%%%%%%%%%%%%%%%%%%%%%%%%%%%%%%%%%%%%%%%

Observe that if the curvature of $\gamma$ exists at $\x_0$, then
for the non-empty valued multi-valued function $\nx:=$"set of
values of all outer unit normal vectors of $\gamma$ at $\x$", we
necessarily have $\# \bfn(\x_0)=1$ and the curvature can be
written as
\begin{equation}\label{curvature}
\kappa(\x_0)=\lim\limits_{\y\to\x_0~\bfv\in \ny} \frac{\arccos
\langle \bfn(\x_0),\bfv \rangle}{|\x_0-\y|},
\end{equation}
where the limit in \eqref{curvature} exists with arbitrary choice
of $\bfv\in \y$ and is independent of this choice.

The next two propositions are well-known.

\begin{proposition}\label{prop:limits} Let $\gamma$ be a planar
convex curve. Recall that \eqref{modcontuv} and \eqref{minoscuv}
is the modulus of continuity and the minimal oscillation of the
multi-valued normal vector function $\nx$ with respect to chord
length, and that \eqref{omegaiv} and \eqref{Omegaiv} stand for the
modulus of continuity and the minimal oscillation of $\nx$ with
respect to arc length. Then for all $\x\in \gamma$ with curvature
$\kappa(\x)\in[0,\infty]$ we have
\begin{equation}\label{kappaomega}
\lim_{\tau\to 0} \frac{\Omega(\tau)}{\tau} = \lim_{\tau\to 0}
\frac{\tilde{\Omega}(\tau)}{\tau} \leq \kappa(\x) \leq
\lim_{\tau\to 0} \frac{\tilde{\omega}(\tau)}{\tau} =\lim_{\tau\to
0} \frac{\omega(\tau)}{\tau}.
\end{equation}
\end{proposition}

\begin{proof} First of all, by definition and the obvious fact
that chord length does not exceed arc length, it follows that
${\Omega(\tau)}\leq {\tilde{\Omega}(\tau)}\leq
{\tilde{\omega}(\tau)} \leq {\omega(\tau)}$. We have already
remarked, that the limits $\lim_{\tau\to 0}
{\tilde{\Omega}(\tau)}/{\tau}$ and $\lim_{\tau\to 0}
{\tilde{\omega}(\tau)}/{\tau}$ exist; moreover, $\lim_{\tau\to 0}
{\tilde{\Omega}(\tau)}/{\tau}=\inf {\tilde{\Omega}(\tau)}/{\tau}
\leq 2\pi/\ell(\gamma)$ is necessarily finite.

On the other hand, let $\tau$ be any fixed value, chosen
sufficiently small, and choose $0\leq s<t<\ell(\gamma)$,
$\gamma(s)=\x$ and $\gamma(t)=\y$ with $|\x-\y|=\tau$ such that
${\Omega}(\tau)=\arccos \langle \bfu,\bfv \rangle$ with some $\bfu
\in \bfn(\x)$, $\bfv\in \bfn(\y)$. Then clearly
$\arg{\bfu}=\tilde{\bfn}_{+}(s)$,
$\arg{\bfv}=\tilde{\bfn}_{-}(t)$, also
$\Omega(\tau)=\tilde{\bfn}_{-}(t)-\tilde{\bfn}_{+}(s)$, and for
all $s<\sigma<t$ we have $\tilde{\bfn}(\sigma)\subset
[\tilde{\bfn}_{+}(s),\tilde{\bfn}_{-}(t)]$. Moreover, putting
$\bnu$ for the normal vector of the chord $\y-\x$, having right
angle with it in the clockwise direction, we also have $\arg(\bnu)
\in [\tilde{\bfn}_{+}(s),\tilde{\bfn}_{-}(t)]$ because
$\y-\x=\int_s^t \t(\sigma) d\sigma= \int_s^t i{\bfn}(\sigma)
d\sigma$, and thus $\arg(\y-\x)\in
[\tilde{\bfn}_{+}(s)+\pi/2,\tilde{\bfn}_{-}(t)+\pi/2] \mod 2\pi$.

Now we compare arc length and chord length. We find
$\tau=|\y-\x|=\int_s^t \langle \bfn(\sigma),\bnu \rangle d\sigma
\geq (s-t)
\cos(\tilde{\bfn}_{-}(t)-\tilde{\bfn}_{+}(s))=(s-t)\cos\Omega(\tau)$,
and, as $\Omega(\tau)=O(\tau)$, we surely have
$\cos(\Omega(\tau))\to 1$ when $\tau\to 0$. It is also clear that
$t-s\to 0$ together with $\tau\to 0$, so for $\tau$ chosen
sufficiently small,
$$
\frac{\Omega(\tau)}{\tau}\geq \frac{\tilde{\Omega}(t-s)}{\tau} =
\frac{t-s}{\tau} \frac{\tilde{\Omega}(t-s)}{t-s} \geq
\frac{t-s}{\tau}(1-\varepsilon) \lim_{\xi\to 0}
\frac{\tilde{\Omega}(\xi)}{\xi} \geq (1-\varepsilon)^2
\lim_{\xi\to 0} \frac{\tilde{\Omega}(\xi)}{\xi}
$$
and it follows that the two limits of the oscillation functions
coincide.

For the modulus of continuity type quantities note that if $\bfn$
is really multivalued, i.e. there exists some point $\x\in\gamma$
where $\bfn(\x)$ consists of more than one vector, then
$\tilde{\bfn}$ attains some closed interval and
$\tilde{\omega}(\tau)$ does not go to $0$ with $\tau$: whence the
arising limits must be $+\infty$. Therefore, it suffices to
consider the case when $\bfn$, i.e. $\tilde{\bfn}$, are
single-valued (and thus $\tilde{\bfn}$ is monotonous and
continuous) functions.

Again, consider a given value $\tau>0$, sufficiently small, and a
pair of extremal points $\x=\gamma(s)$ and $\y=\gamma(t)$ with
$0\leq s<t<\ell(\gamma)$ such that $\tau=|\y-\x|$ and
$\omega(\tau)=\arccos \langle
\bfn(\x),\bfn(\y)\rangle=\tilde{\bfn}(t)-\tilde{\bfn}(s)$. As
above, for all $s<\sigma<t$ we have $\tilde{\bfn}(\sigma)\subset
[\tilde{\bfn}(s),\tilde{\bfn}(t)]$. Moreover,
$\tau=|\y-\x|=\int_s^t \langle \bfn(\sigma),\bnu \rangle d\sigma
\geq (t-s)
\cos(\tilde{\bfn}(t)-\tilde{\bfn}(s))=(t-s)\cos\omega(\tau)$, and,
as $\omega(\tau)\to 0$, we surely have $\cos(\omega(\tau))\to 1$
when $\tau\to 0$. (Observe that we already have $t-s\to 0$
together with $\tau\to 0$ -- however, we do not need it here.) At
last, we find for $\tau$ chosen sufficiently small,
$$
\frac{\omega(\tau)}{\tau}\leq \frac{t-s}{\tau}
\frac{\tilde{\omega}(t-s)}{t-s} \leq \frac1{\cos{\omega(\tau)}}
\sup_{\xi}\frac{\tilde{\omega}(\xi)}{\xi} \leq (1+\varepsilon)
\lim_{\xi\to 0}\frac{\tilde{\omega}(\xi)}{\xi}.
$$
It follows that the leftmost and rightmost limits in
\eqref{kappaomega} exist and are equal to the corresponding limits
with respect to arc length. Therefore, it suffices to prove the
inequalities involving $\kappa(\x)$ for the quantities $\omega$
and $\Omega$ only.

Clearly, $\Omega(\bfn, |\x-\y|) \leq \arccos \langle \bfu,\bfv
\rangle \leq \omega(\bfn,|\x-\y|)$ for all $\bfu \in \nx$,
$\bfv\in \ny$. Putting $\tau=|\y-\x|$, and recalling that $\nx$ is
unique by condition of existence of $\kappa(\x)$, we obtain
$$
\lim_{\tau \to 0} \frac{\Omega(\bfn,\tau)}{\tau}\leq
\Bigg(\kappa(\x)=\Bigg) \lim\limits_{\y\to\x~\bfv\in \ny}
\frac{\arccos \langle \nx,\bfv \rangle}{|\x-\y|} \leq \lim_{\tau
\to 0} \frac{\omega(\bfn,\tau)}{\tau}.
$$
\end{proof}

In the following proposition $\arccos$ will denote the branch with
values in $[0,\pi]$.

\begin{proposition}\label{prop:bounds} Let $\gamma$ be a closed convex
curve, and \eqref{modcontuv} and \eqref{minoscuv} be the modulus
of continuity and the minimal oscillation of the (in general,
multi-valued) unit normal vector function $\nx$.
\begin{enumerate}
    \item[({\it i})] If the curvature exists and is bounded from above by
    $\kappa_0$ all over $\gamma$, then there exists a bound $\tau_0>0$ so
    that for any two points $\x,\y\in\gamma$ with $|\x-\y|\leq \tau \leq
    \tau_0$ we must have $\omega(\bfn,\tau)<\pi/2$ and $\arccos \langle
    \nx,\ny \rangle \leq \kappa_0 {\tau}/{\cos(\omega(\bfn,\tau))}$. Thus we
    also have $\omega(\bfn,\tau) \leq \kappa_0\tau/ \cos(\omega(\bfn,\tau))$
    for $\tau\leq \tau_0$.
    \item[({\it ii})] If the curvature $\kappa(\x)$ exists (linearly, that
    is, according to arc length parametrization) almost everywhere, and
    is bounded from below by $\kappa_0$ (linearly)
    almost everywhere on $\gamma$, then for any two points
    $\x,\y\in\gamma$ with $|\x-\y|\geq\tau$ and for all
    $\bfu \in \nx$,$\bfv \in \ny$ we have
    $\arccos \langle \bfu,\bfv \rangle \geq \kappa_0 \tau$
    and hence $\Omega(\bfn,\tau)\geq \kappa_0\tau$.
\end{enumerate}
\end{proposition}
\begin{proof} Consider first ($i$). In this case $\bfn$ is a
single-valued function. Recall that $\alpha$ stands for the
\emph{tangent angle} function, and so with $\x=\gamma(s_0)$ and
$\y=\gamma(t)$ $\arccos \langle \nx,\ny \rangle
=\alpha(t)-\alpha(s_0)$, supposing that on the counterclockwise
closed arc $\x\y$ of $\gamma$ the rotation of the outer unit
normal vector is at most $\pi$. (In case of the rotation exceeding
$\pi$, the complementary arc must have rotation below $\pi$, and
considering the negatively oriented curve, i.e. a reflection of
$\gamma$, we can conclude the same way.) Since the curvature is
just $\kappa=\alpha'$ ($\alpha$ written in arc length
parametrization), by condition $\alpha$ is an \emph{everywhere
differentiable function} (with respect to arc length). Thus we can
apply the Lagrange mean value theorem to find some parameter $u\in
(s_0,t)$ satisfying
$$
\alpha(t)-\alpha(s_0)= \alpha'(u) (t-s_0).
$$
Now we can apply the condition $\alpha'=\kappa\leq \kappa_0$ to
get
\begin{equation}\label{omegakappa}
\arccos \langle \nx,\ny \rangle \leq \kappa_0 (t-s_0).
\end{equation}
It remains to estimate the arc length $t-s_0$ in function of
$\tau$.

Let now $\x,\y$ be two arbitrary points of $\gamma$ and consider
the counterclockwise arc of $\gamma$ between these points. Let us
suppose that this arc has total curvature less than $\pi/2$. Since
$\kappa$ exists and is bounded everywhere by $\kappa_0$, a
standard compactness argument yields $\omega(\bfn,\tau) < \pi/2$
for $\tau:=|\y-\x| \leq \tau_0$. As now $\nx$ is single-valued, we
have $\arg \t(\x) =\alpha(s_0)$ with the unique tangent vector at
$\x$, and we can write
\begin{align*}
t-s_0& =\int_{s_0}^t 1 ds \leq \int_{s_0}^t \frac{\cos(
\alpha(s)-\alpha(s_0))}{\cos (\alpha(t)-\alpha(s_0))} ds =
\frac{1}{\cos(\alpha (t)-\alpha(s_0))} \int_{s_0}^t \langle
\dot{\gamma}(s); \t(\x) \rangle ds \\
& \leq \frac {1}{\cos \omega(\bfn,|\y-\x|)} \left\langle
\int_{s_0}^t \dot{\gamma}(s)ds ; \t(\x) \right\rangle = \frac
{\left\langle \y-\x; \t(\x) \right\rangle}{\cos \omega(\bfn,\tau)}
\leq \frac{\tau}{\cos \omega(\bfn,\tau)}.
\end{align*}
On combining this with \eqref{omegakappa}, the assertion ({\it i})
follows.

To prove ({\it ii}) we still can use that $\alpha$ is a monotonic
function, hence is almost everywhere differentiable and, as
detailed above, for any $\bfu \in \nx$, $\bfv \in \ny$ we have
$$
\arccos \langle \bfu,\bfv \rangle =\arg \bfu - \arg \bfv \geq
\tilde{\bfn}_{-}(t)-\tilde{\bfn}_{+}(s_0) =
\alpha_{-}(t)-\alpha_{+}(s_0) \geq \int_{s_0}^t \alpha'(s) ds \geq
\kappa_0 (t-s_0)
$$
by condition of $\kappa = \alpha' \geq \kappa_0$ (linearly) a.e.
on $\gamma$. (As above, we may assume that $\arg{\bfu}-\arg{\bfv}$
does not exceed $\pi$, as otherwise we may consider the
complementary arc, i.e. the reflected curve with respect to the
line of $\x$ and $\y$, e.g.) It is obvious that the arc length of
$\gamma$ between $\x$ and $\y$ is at least the distance of $\x$
and $\y$, hence the assertion follows.
\end{proof}

Rotations of $\CC=\RR^2$ about the origin $O$ by the
counterclockwise measured (positive) angle $\varphi$ will be
denoted by $U_{\varphi}$, that is,
\begin{equation}\label{Ualpha}
U_{\varphi}=\left(%
\begin{array}{cc}
  \cos\varphi & -\sin\varphi \\
  \sin\varphi & \cos\varphi \\
\end{array}%
\right).
\end{equation}
We denote $T$ the reflection to the $y$-axis, i.e. the linear
mapping defined by $\left(%
\begin{array}{cc}
  -1 & 0 \\
  0 & 1 \\
\end{array}%
\right) $.

\begin{definition}[{\bf Mangled $n$-gons}]\label{def:M} Let $2\leq k\in\NN$ and
put $n=4k-4$, $\varphi^{*}:=\frac{\pi}{2k}$. We define the
\emph{standard mangled $n$-gon} as the convex $n$-gon
\begin{equation}\label{Mkdef}
M_k:=\con\{A_1,\dots,A_{k-1},A_{k+1},
\dots,A_{2k-1},A_{2k+1},\dots,A_{3k-1},A_{3k+1},\dots, A_{4k-1}\},
\end{equation}
of $n=4k-4$ vertices with
\begin{equation}\label{MAmdef}
A_m:= \left(\sum_{j=1}^m
\cos(j\varphi^{*})-\sum_{\ell=1}^{\Floor{m/k}} \cos(\ell
k\varphi^{*}) ,\sum_{j=1}^m
\sin(j\varphi^{*})-\sum_{\ell=1}^{\Floor{m/k}} \sin(\ell
k\varphi^{*})\right),
\end{equation}
where $m\in \{1,\dots,4k\}\setminus\{k,2k,3k,4k\}$. That is, we
consider a regular $4k$-gon of unit sides, but cut out the middle
"cross-shape" (i.e., the union of two rectangles which are the
convex hulls of two opposite sides of the regular $4k$-gon, these
pairs of opposite sides being perpendicular to each other) and
push together the left over four quadrants (i.e., shift the
vertices $A_{\ell k}$ to the position of $A_{\ell k-1}$
consecutively to join the remaining sides of the polygon. Observe
that taking $A_0:=O$, the same formula \eqref{MAmdef} is valid
also for $A_0:=O=A_{4k}=A_{4k-1}$ and $A_{\ell k}=A_{\ell k-1}$,
$\ell=1,2,3,4$, showing how the vertices of the regular $4k$-gon
were moved into their new positions.)

Now let $\tau>0$, $\al\in\RR$, $\x\in\RR^2$ and
$\varphi\in(0,\pi/4]$ be arbitrary. Take $k:=
\Floor{\frac{\pi}{2\varphi}}$, so that $\varphi^{*}:=
\frac{\pi}{2k}\geq \varphi$.

Then we write $M(\varphi):=M_k$, and, moreover, we also define
\begin{equation}\label{Mdef}
M(\x,\al,\varphi,\tau):=M(\x,\al,\varphi^{*},\tau):=U_{\al}\left(\tau
M_k\right)+\x,
\end{equation}
that is, the copy shifted by $\x$ of the $4k-4$-gon obtained by
dilating $M(\varphi)=M_k$ from $O=A_0=A_{4k-1}$ with $\tau$ and
rotating it counterclockwise about $O$ by the angle $\al$.
\end{definition}

E.g. if $\varphi\in(\pi/6,\pi/4]$, then $k=2$,
$\varphi^{*}=\pi/4$, $n=4$, and $M_2$ is just a unit square, its
side lines having direction tangents $\pm 1$ and having its lowest
vertex at $O$. It is the left over part, pushed together, of a
regular octagon of unit side length, when the middle cross-shape
is removed from its middle.

It is easy to see that the inradius $\rho(\varphi)$ and the
circumradius $R(\varphi)$ of $M(\varphi)=M(\varphi^{*})=M_k$ are
\begin{align}\label{Mradius}
\begin{cases} \qquad r(\varphi)&=\frac12
\left\{\cot\frac{\pi}{4k}-\sqrt{2}\cos\left(\frac{1-(-1)^k}{8k}\pi\right)\right\},
\\
\qquad R(\varphi)& =\frac 12
\left\{\cot\frac{\pi}{4k}-1\right\},\qquad\qquad\qquad
\end{cases}
\qquad\qquad \left(k:= \Floor{\frac{\pi}{2\varphi}}\right),
\end{align}
respectively.

Similarly to the \emph{mangled $n$-gons} $M_k$, we also define the
\emph{fattened $n$-gons} $F_k$.

\begin{definition}[{\bf Fattened $n$-gons}]\label{def:F} Let $k\in\NN$ and
put $n=4k$, $\varphi^{*}:=\frac{\pi}{2k}$. We first define the
\emph{standard fattened $n$-gon} as the convex $n$-gon
\begin{equation}\label{Fkdef}
F_k:=\con\{A_1,\dots,A_{k-1},A_k, A_{k+1},
\dots,A_{4k-1},A_{4k}\},
\end{equation}
of $n=4k$ vertices with
\begin{equation}\label{FAmdef}
A_m:= \left(\sum_{j=1}^m
\cos(j\varphi^{*})+\sum_{\ell=0}^{\Floor{m/k}} \cos(\ell
k\varphi^{*}) ,\sum_{j=1}^m
\sin(j\varphi^{*})+\sum_{\ell=0}^{\Floor{m/k}} \sin(\ell
k\varphi^{*})\right).
\end{equation}
That is, we consider a regular $4k$-gon , but fatten the middle
"cross-shape" to twice as wide, and move the four quadrants to the
corners formed by this width-doubled cross (i.e., shift the
vertices $A_{\ell k}$ to the position of $A_{\ell k-1}+2 (A_{\ell
k}-A_{\ell k-1})$ consecutively to join the remaining sides of the
polygon). Observe that $A_{4k-1}=(-1,0)$ and $A_{4k}=(1,0)$.

Let $\tau>0$, $\al\in\RR$, $\x\in\RR^2$ and $\varphi\in(0,\pi)$ be
arbitrary. Now we take $k:= \Ceil{\frac{\pi}{2\varphi}}$, whence
$\varphi^{*}:= \frac{\pi}{2k}\leq \varphi$.

Then we write $F(\varphi):=F_k$, and, moreover, we also define
\begin{equation}\label{Pdef}
F(\x,\al,\varphi,\tau):=F(\x,\al,\varphi^{*},\tau):=U_{\al}\left(\tau
F_k\right)+\x,
\end{equation}
that is, the copy shifted by $\x$ of the $4k$-gon obtained by
dilating $F(\varphi)=F_k$ from $O$ with $\tau$ and rotating it
counterclockwise about $O$ by the angle $\al$.
\end{definition}

E.g. if $\varphi\geq \pi/2$, then $k=1$, $\varphi^{*}=\pi/2$,
$n=4$, and $F_4$ is just the square spanned by the vertices (1,0),
(1,2), (-1,2), (-1,0) and having sides of length 2.

Observe that using the usual Minkowski addition, we can represent
the connections of these deformed $n$-gons and the regular $n$-gon
easily. Write $Q_n$ for the regular $n$-gon placed symmetrically
to the $y$-axis but above the $x$-axis with $O\in\partial Q_n$ a
midpoint (hence not a vertex) of a side of $Q_n$. (This position
is uniquely determined.) Also, denote the standard square as
$S:=Q_4:=\con\{(1/2,0);(1/2,1);(-1/2,1);(-1/2,0)\}$. Then we have
$M_k+S=Q_{4k}$ and $Q_{4k}+S=F_k$.

It is also easy to see that the inradius $\mathfrak r(\varphi)$
and the circumradius ${\mathfrak R}(\varphi)$ of
$F(\varphi)=F(\varphi^{*})$ are
\begin{equation}\label{Finradius}
{\mathfrak r}(\varphi)=\frac12 \cot\frac{\pi}{4k}+\frac12
\qquad\qquad\quad \left(k:= \Ceil{\frac{\pi}{2\varphi}}\right),
\end{equation}
and
\begin{equation}\label{Fcircumradius}
{\mathfrak R}(\varphi)=\begin{cases}
\frac1{2\sin\frac{\pi}{4k}}+\frac1{\sqrt{2}} \qquad\qquad
&\textrm{if}~~2\nmid k \\
\sqrt{\frac12 + \frac1{4 \sin^2\frac{\pi}{4k}}
+\frac1{\sqrt{2}}\cot\frac{\pi}{4k}} ~~&\textrm{if}~~2\mid k
\end{cases}\qquad \left(k:= \Ceil{\frac{\pi}{2\varphi}}\right),
\end{equation}
respectively.

The actual values of the above in- and circumradii in
\eqref{Mradius}, \eqref{Finradius}, \eqref{Fcircumradius} are not
important, but observe that for $\varphi\to 0$, or, equivalently,
for $k\to\infty$, we have the asymptotic relation $r(\varphi)\sim
R(\varphi)\sim \mathfrak{r}(\varphi) \sim
\mathfrak{R}(\varphi)\sim \frac{1}{\varphi}$.

\section{The discrete Blaschke theorems}\label{sec:discreteresults}

\begin{theorem}\label{th:inscribed} Let $K\subset \CC$ be a convex
body and $0<\varphi<\pi/4$. Denote $\bfn$ the (multivalued)
function of outer unit normal(s) to the closed convex curve
$\gamma:=\partial K$ and assume that $\omega(\bfn,\tau)\leq
\varphi <\pi/4$. Put $k:=\Floor{\frac{\pi}{2\varphi}}$. If
$\bfx\in\partial K=\gamma$, and
$\bfn_0=(\sin\alpha,-\cos\alpha)\in \nx$ is outer unit normal to
$\gamma$ at $\x$, then $M(\x,\al,\varphi,\tau) \subset K$.
\end{theorem}

\begin{proof} Because $\varphi \leq \varphi^{*}:=\pi/(2k)$ and
$M(\x,\al,\varphi,\tau)=M(\x,\al,\varphi^{*},\tau)$, it suffices
to present a proof for the case when
$\varphi=\varphi^{*}=\frac{\pi}{2k}$.

Applying simple transformations we may reduce to the case $\x=O$
and $\alpha=0$, $\tau=1$. With these restrictions we are to prove
$M_k\subset K$, where $O\in K=\partial \gamma$, $(0,-1)$ is an
outer normal to $K$ at $O$, and $\omega(\bfn,1)\leq \varphi$.
Denote $P=(a,b)$ the first point, along $\gamma$ following $O$
counterclockwise, satisfying that $(1,0)$ is outer normal to $K$
at $P$. Clearly, then $\gamma$ can be parameterized with the
$x$-values along the $x$-axis so that $\gamma(x)=(x,g(x))$ for
values $x\in [0,a]$, and $g$ is a convex function on $[0,a]$.

Consider $A_m=(a_m,b_m)$ defined in \eqref{MAmdef} for
$m=0,\dots,k-1$, putting here $A_0:=A_{4k-1}=O$, and consider the
function
\begin{equation}\label{fdef}
f(x):=\begin{cases} \dots \\ (x-a_{m-1})\tan \frac{m\pi}{2k}
+a_{m-1}
\quad (a_{m-1}\leq x \leq a_m)   \\
\dots
\end{cases}\quad (m=1,\dots,k-1).
\end{equation}
Moreover, denote the broken line joining $O=A_0,A_1,\dots,A_{k-1}$
as $L$, that is,
\begin{equation}\label{Ldef}
L:= \{ (x,f(x))~:~ 0\leq x \leq a_{k-1} \}.
\end{equation}
\begin{lemma}\label{lemma:omega} Let $O\in \gamma=\partial K$,
$(0,-1)$ is outer normal to $K$ at $O$, and $\omega(\bfn,1)\leq
\varphi=\frac{\pi}{2k}$. With the notations above, we have
\begin{enumerate}
\item[\rm (i)] $a\geq a_{k-1}=\frac12 (\cot\frac{\pi}{4k} -1)(=R(\varphi))$.
\item[\rm (ii)] $0\leq g(x) \leq f(x)$ for all $x\in [0,a_{k-1}]$.
\item[\rm (iii)] $g'_{\pm}(x) \leq f'_{\pm}(x)$ for all $x\in (0,a_{k-1})$
and $g'_{+}(0) \leq f'_{+}(0)$, $g'_{-}(a_{k-1}) \leq
f'_{-}(a_{k-1})$.
\item[\rm (iv)] $b:=g(a)\geq a_{k-1}$.
\item[\rm (v)] $L\subset K$.
\end{enumerate}
\end{lemma}
\begin{proof} Let $a^{*}:=\min(a,a_{k-1})$. We argue by induction on $m$,
where $m=1,\dots k-1$, and the inductive assertions will comprise
\begin{enumerate}
\item[\rm (i')] $a^{*}\geq a_m$;
\item[\rm (ii')] $0\leq g(x) \leq f(x)$ for all $x\in
[a_{m-1},a_{m}]$;
\item[\rm (iii')] $g'_{\pm}(x) \leq f'_{\pm}(x)$ for all $x\in
(a_{m-1},a_{m})$ and $g'_{+}(a_{m-1}) \leq f'_{+}(a_{m-1})$,
$g'_{-}(a_{m}) \leq f'_{-}(a_{m})$.
\end{enumerate}
Clearly, if we show this for all $m=1,\dots ,k-1$ then (i)-(iii)
of the Lemma will be proved.

Let us start with $m=1$. Since $O=(0,0)\in\gamma$, we have
$g(a_{m-1})=g(0)=0\leq f(a_{m-1})=f(0)=0$. Let $S:=\{x\in [0,a_1]
~:~ g|_{[0,x]} \leq f|_{[0,x]} \}$. Clearly, $S$ is a (possibly
degenerate) closed interval with left end point 0, say $[0,X]$.
Our aim is to prove that $S=[0,a_1]$. Clearly, if $X=a_1\in S$,
then a relative neighborhood of $X$ belongs to $S$, too. We prove
the same thing for any other $X\in S$. Observe that the distance
of $O=A_0$ and $A_1$ is 1, and all other points of the triangle
$\Delta:=\Delta (O,(a_1,0),A_1)$ are closer than 1 to $O=A_0$. In
particular, in case $X<a_1$, both $\{(x,g(x))~:~ 0\leq x \leq X\}$
and also a small neighborhood of $(X,g(X))\in \Delta$ is also
closer to $O$ than 1. It follows that the continuous curve
$\gamma$ runs in the 1-neighborhood of $O$ even in an
appropriately small neighborhood of $(X,g(X))\in\gamma$.
Therefore, by assumption on the change of the normal to $\gamma$,
the vector $(0,1)$ in the counterclockwise taken angular region
between the left and right hand side half-tangents (oriented
according to the positive orientation of $\gamma$) to $\gamma$ at
$O$, cannot rotate over $(\cos \varphi, \sin \varphi)$ along
$\{(x,g(x))~:~ 0\leq x \leq X+\eta\}$ for some positive value of
$\eta$. That is, $a^{*}\geq X+\eta$ and the representation
$\gamma(x)=(x,g(x))$ is valid for $x\in [0,X+\eta]$; moreover,
$g'_{\pm}(x) \leq \tan \varphi$ for all $x\in [0,X+\eta]$. In
conclusion, $g(x)=\int_0^x g'(\xi) d\xi \leq x \cdot \tan \varphi
= f(x)$ for all $x\in [0,X+\eta]$. As a result, we find that $S$
is relatively open. As it is also closed and nonempty, it is the
whole interval $[0,a_1]$. This proves (i') and (ii') for $m=1$,
and (iii') follows from the fact that $\{(x,g(x))~:~ 0\leq x \leq
a_1 \}\subset \Delta$ and thus the distance of any point of
$\{(x,g(x))~:~ 0\leq x \leq a_1 \}$ from $O$ is at most 1.

We proceed by induction. Let $1<m<k$ and assume the assertion for
all $m'<m$. Then from the inductive hypothesis $a^{*}\geq
a_{m-1}$, $\mu:=\mu_{m-1}:=g'_{-}(a_{m-1}) \leq
f'_{-}(a_{m-1})=\tan ((m-1)\varphi)$ and $g(a_{m-1}):=y_{m-1} \leq
f(a_{m-1})=b_{m-1}$. Consider now the function
$h(x):=y_{m-1}+(x-a_{m-1})\tan(m\varphi)$ (defined for $x\in
I_m:=[a_{m-1},a_m]$), denote the points
$P_{m-1}:=(a_{m-1},y_{m-1})$ and $P_m:=(a_m,h(a_m))$, and define
the triangle $\Delta:=\Delta_m:=\Delta(P_{m-1},(a_m,y_{m-1}+\mu
\cos (m\varphi)),P_m)$. Then $h=f|_{I_m}-(b_{m-1}-y_{m-1})\leq
f|_{I_m}$, and $h'_{\pm}=f'_{\pm}$ on $I_m$.

Our aim now is to show that $\gamma$ proceeds inside
$\Delta=\Delta_m$. Observe that for points $Q$ inside $\Delta$ we
have $|Q-P_{m-1}|\leq 1$, with equality holding only if $Q=P_m$.
Therefore, for $Q\in\gamma \cap \Delta$ the right half-tangent
direction to $\gamma$ cannot exceed $\arctan \mu + \varphi\leq
m\varphi$, and, moreover, the same properties hold even for a
relative neighborhood of $Q$ on $\gamma$ if $Q\ne P_m$.

So we proceed similarly to the case $m=1$. It is obvious that
$a^{*}>a_{m-1}$ as $\gamma$ proceeds between slopes $\mu$ and
$\tan (m\varphi)$ in the 1-neighborhood of $P_{m-1}$. Take
$S:=S_m:=\{x\in I_m~:~ g|_{[a_{m-1},x]}\leq h|_{[a_{m-1},x]} \}$.
Again, by continuity of $g$ and linearity of $h$ $S$ is a closed
interval $[a_{m-1},X]$, say. Also, if $X=a_m$, then $S=I_m$ and so
$S$ is relatively open in $I_m$, and if $a_m\ne X\in S$, then
$(X,g(X))\in\gamma\cap \Delta$ has a small neighborhood where
$\gamma$ stays within the 1-neighborhood of $P_{m-1}$, therefore
its slope is below $\tan(\arctan \mu + \varphi)$ and
$\gamma(x)=(x,g(x))$ extends even until some $X+\eta$; moreover,
$a^{*}\geq X+\eta$ and $\mu \leq g'_{\pm} \leq \tan( \arctan \mu +
\varphi)\leq \tan(m\varphi)$ holds all over $[a_{m-1},X+\eta]$
(where for $a_{m-1}$ and $X+\eta$ we claim only the inequalities
for $g_{+}$ and $g_{-}$, resp.), proving
\begin{equation}\label{mthequation}
\mu (x-a_{m-1})+y_{m-1} \leq g(x) =\int_{a_{m-1}}^x g'(\xi)
d\xi+y_{m-1} \leq \tan (m\varphi)(x-a_{m-1})+y_{m-1} =h(x)
\end{equation}
for all $a_{m-1}\leq x \leq X+\eta$. That is, $\gamma$ stays
inside $\Delta$ and $S$ contains a small neighborhood of $X$, too.
It follows that $S\ne \emptyset$ is open and closed, while $I_m$
is connected, thus $S_m=I_m$ and \eqref{mthequation} holds true
even for the whole of $I_m$. This proves (i')-(iii'), hence
(i)-(iii) of the Lemma.

Applying the above we find $a>a_{k-1}$. However, a simple argument
immediately gives also $b>a_{k-1}$, too. Indeed, it suffices to
consider the new curve
$\widehat{\gamma}:=T(U_{-\pi/2}(\gamma-(a,b))$, obtained from
$\gamma$ first shifting it by $-P=-(a,b)$, then rotating it by
$-\pi/2$ about $O$, and finally reflecting it at the $y$-axis.
This shows (iv).

Also, applying the Lemma for the reflected curve
$\widetilde{\gamma}$ of $\gamma$ with respect to the $y$-axis
gives a similar result for the part of $\gamma$ towards the
"negative $x$-direction". That is, we find that $\gamma$ joins the
points $\widetilde{P}=(\widetilde{a},\widetilde{b})$ and $P=(a,b)$
with (some of their) outer unit normals $(-1,0)$ and $(1,0)$,
respectively, so that the part strictly between these points (and
containing $O$) does never have horizontal normals, and we have a
parametrization $\gamma(x)=(x,g(x))$ for all $\widetilde{a}\leq
x\leq a$ with $\widetilde{a}\leq -a_{k-1}$, $a \geq a_{k-1}$, and
$0\leq g(x) \leq f(|x|)$, $|g'_{\pm}(x)|\leq f'_{\pm}(|x|)$ for
all $x\in[-a_{k-1},a_{k-1}]$. Note that we also have
$\widetilde{b}\geq a_{k-1}$, as above.

Finally let us show (v). Consider any point $(x,f(x))$ of $L$,
where $x\in[0,a_{k-1}]$. There is a vertical line $\ell$ through
it that intersects $K$ in a vertical chord $C$ of $K$. The lower
endpoint of $C$ is $(x,g(x)$. The upper endpoint of $C$ has second
coordinate at least $\min\{b,\tilde{b}\}\geq a_{k-1}$. Hence the
point $(x,f(x))$ lies on the chord $C$ of $K$, whence in $K$. This
proves (v).

\end{proof}

\emph{Continuation of the proof of Theorem \ref{th:inscribed}.}
From the above argument --  or just reflecting $L$ to the $y$-axis
--  it is immediate that also the broken line $\widetilde{L}$
joining $A_{3k+1},\dots,A_{4k-1}=O$ in this order that lies on the
boundary of $M_k$ belongs to $K$, too. We are left with the upper
part joining
$A_{k-1},A_{k+1},\dots,A_{2k-1},A_{2k+1},\dots,A_{3k-1}$. Let
\begin{equation}\label{Lpmdef}
L^{+}:=[A_{k-1},A_{k+1}] \cup \dots \cup [A_{2k-2},A_{2k-1}],
\quad L^{-}:=[A_{2k-1},A_{2k+1}] \cup \dots \cup
[A_{3k-2},A_{3k-1}].
\end{equation}

Next, let us apply the Lemma to the curve from $P$ onwards in the
counterclockwise sense. That is, take
$K_{+}:=U_{-\pi/2}(K-P)$ (with $U_{\alpha}$ as defined in
\eqref{Ualpha} ) and $\gamma_{+}:=U_{-\pi/2}(\gamma-P)$ and check
that $O\in \gamma_{+}$ and also $\gamma_{+}$ has an outer normal
$(0,-1)$ at $O$; moreover, the same estimate on the modulus of
continuity of the normal holds for $\gamma_{+}$. Thus we obtain
that $L\subset K_{+}$, that is, $U_{\pi/2}(L)+P\subset K$. Observe
$U_{\pi/2}(L)=L^{+}-(a_{k-1},a_{k-1})$, which entails
$L^{+}+(a-a_{k-1},b-a_{k-1})\subset K$. It suffices to say that
$L^{+}+(u,v)\subset K$ with $u,v\geq 0$.

Very similarly (or from this and using reflection) we also obtain
$L^{-}+(p,q)\subset K$ with $p\geq 0$ and $q\leq 0$.

We claim that $A_{2k-1}=(0,2a_{k-1})\in K$. Indeed,
$A_{2k-1}+(0,-2a_{k-1})=O\in K$ and $A_{2k-1}+(u,v)\in
L^{+}+(u,v)\subset K$, $A_{2k-1}+(p,q)\in L^{-}+(p,q)\subset K$,
and the convex hull of the vectors $(0,-2a_{k-1})$, $(u,v)$ and
$(p,q)$ contains $(0,0)$, hence by convexity $A_{2k-1}\in K$.

Now, for showing $L^{+}\subset K$, recall that $A_{k-1}\in L
\subset K$, $A_{2k-1}\in K$, and $L^{+}+(u,v) \subset K$. So it
remains to see that $L^{+}$ is in the convex hull of its two
endpoints and the set $L^{+}+(u,v)$ whenever $u,v\geq 0$.
Similarly one obtains $L^{-}\subset K$. That concludes the proof.
\end{proof}

An even stronger version can be proved considering the modulus of
continuity $\tilo$ with respect to arc length. We thank this
sharpening to Endre Makai, who kindly called our attention to this
possibility and suggested the crucial Lemma \ref{lemma:triangle}
for the proof.

\begin{theorem}\label{th:arclength} Let $K\subset \CC$ be a
planar convex body and $0<\varphi<\pi/4$. Denote $\bfn$ the
(multivalued) function of outer unit normal(s) to the closed
convex curve $\gamma:=\partial K$ and assume that $\tilo(\tau)\leq
\varphi <\pi/4$. Put $k:=\Floor{\frac{\pi}{2\varphi}}$. If
$\bfx\in\partial K=\gamma$, and
$\bfn_0=(\sin\alpha,-\cos\alpha)\in\nx$ is outer unit normal to
$\gamma$ at $\x$, then $M(\x,\al,\varphi,\tau) \subset K$.
\end{theorem}

\begin{proof} We can repeat the argument yielding Theorem
\ref{th:inscribed} with the only change that in the inductive
argument for proving Lemma \ref{lemma:omega}, we have to use twice
(once for the case $m=1$ to start the inductive argument, and once
for general $m$) a slightly sharper geometric assertion to ensure
that even in this setting the boundary curve $\gamma$ of $K$ will
again proceed in the triangles $\Delta:=\Delta(O,(a_1,0),A_1)$ and
$\Delta:=\Delta_m:=\Delta(P_{m-1},(a_m,y_{m-1}+\mu \cos
(m\varphi)),P_m)$.

The general situation will be covered by the following lemma.

\begin{lemma}\label{lemma:triangle} Let $\Delta=\Delta(P,Q,R)$ be
the right- or obtuse triangle spanned by the points $P=(a,p)$,
$Q=(b,q)$ with $b>a$ and $q\geq p$, and $R=(b,r)$ with $r>q$.
Denote $\rho:=\sqrt{(b-a)^2+(r-p)^2}$ the length of the longest
side of $\Delta$, and let $\mu:=(q-p)/(b-a)$, resp.
$\nu:=(r-p)/(b-a)$ be the slopes of sides $PQ$ and $PR$,
respectively, with corresponding angles $\psi:=\arctan \mu$ and
$\lambda:=\arctan \nu$. Denote $\varphi:=\lambda-\psi$ the angle
of $\Delta$ at $P$.

Let $\Gamma$ be a convex curve of arc length $\rho$, connecting
the points $P$ and $N=(n,s)$ and having all its tangent vectors at
all points of $\Gamma$ (including the right half tangent at $P$
and the left half tangent at $N$) with angles between $\psi$ and
$\psi+\varphi=\lambda$. Then $n\geq b$, the only possibility for
equality is when $N=R$, otherwise $n>b$ and $\Gamma$ intersects
the vertical side of $\Delta$ at a mesh point $M=(b,m)$ with
$q\leq m <r$. Moreover, $s \in [p+\mu(n-a),r]$.
\end{lemma}
\begin{proof} By convexity, the non-empty valued, multivalued
tangent vector function $\t$ along $\Gamma$ is continuous (in the
weak sense) and nondecreasing, and also we have for the
multi-valued tangent angle function
$\widehat{\alpha}(\sigma)=\arg{\t}(\sigma)\in [\psi,\lambda]$ for
all $0\leq \sigma\leq \rho$, i.e. all along $\Gamma$. Therefore,
$$
(n-a,s-p)=N-P=\int_0^{\rho} \t(\sigma) d\sigma = \int_0^{\rho}
(\cos(\alpha(\sigma)),\sin(\alpha(\sigma))d\sigma,
$$
now neglecting the linearly $0$-measure set of points where $\t$
or $\widehat{\alpha}$ is indeed multi-valued. By condition we find
$n-a\geq \rho \cos \lambda=b-a$ and equality would imply $\cos
\alpha(\sigma)=\cos\lambda$ a.e., that is $\Gamma=[P,R]$.
Otherwise by $|\Gamma|=\rho$ and $n>b$ we surely have $s<r$.
Finally, for the directional tangent of the chord $[P,N]$ we see
$$
\frac{s-p}{n-a}=\frac{\int_0^\rho
\sin\alpha(\sigma)d\sigma}{\int_0^\rho \cos\alpha(\sigma)d\sigma}
\geq \frac{\int_0^\rho \sin \psi d\sigma}{\int_0^\rho \cos \psi
d\sigma} =\tan \psi =\mu.
$$
The Lemma follows.
\end{proof}
In applying the above lemma we start with the observation that by
condition $\tilo(1)\le \varphi$, the tangent angle of $\gamma $
can increase at most $\varphi$ along the part of $\gamma$ which is
closer than 1 to the point $O$ (in case $m=1$) or to $P_{m-1}$ (in
case of the inductive step with general $m$). Therefore,
proceeding along $\gamma$ with arc length 1 and denoting this arc
of $\gamma$ as $\Gamma$, we will have a convex curve, with tangent
angles between $\psi$ and $\psi+\varphi$, as in the above lemma.
Therefore, Lemma \ref{lemma:triangle} will ensure that the
argument goes through for proving the corresponding version of
Lemma \ref{lemma:omega} with $\omega(\bfn,1)$ replaced by
$\tilo(1)$. Otherwise the argument is the same.
\end{proof}

\begin{theorem}\label{th:outscribed}
Let $K\subset \CC$ be a (planar) convex body and $\tau>0$. Denote
$\bfn$ the (multivalued) function of outer unit normal(s) to the
closed convex curve $\gamma:=\partial K$ and assume that
$\Omega(\bfn,\tau)\geq \varphi$. Take
$k:=\Ceil{\frac{\pi}{2\varphi}}$. If $\bfx\in\partial K=\gamma$,
and $\bfn_0=(\sin\alpha,-\cos\alpha)\in \nx$ is normal to $\gamma$
at $\x$, then $F(\x,\al,\varphi,\tau) \supset K$.
\end{theorem}

\begin{proof} Because $\varphi \geq \varphi^{*}:=\frac{\pi}{2k}$ and
$F(\x,\al,\varphi,\tau)=F(\x,\al,\varphi^{*},\tau)$, it suffices
to present a proof for the case when
$\varphi=\varphi^{*}=\frac{\pi}{2k}$.

Applying simple transformations we may reduce to the case $\x=O$
and $\alpha=0$, $\tau=1$. With these restrictions we are to prove
$F_k\supset K$, where $O\in K=\partial \gamma$, $(0,-1)$ is an
outer normal to $K$ at $O$, and $\Omega(\bfn,1)\geq \varphi$.

Denote $P=(a,b)$ the first point counterclokwise after $O$, along
$\gamma$, satisfying that $(1,0)$ is an outer unit normal to $K$
at $P$. Clearly, then $\gamma$ can be parameterized with the
$x$-values along the $x$-axis so that $\gamma(x)=(x,g(x))$ for
values $x\in [0,a]$, and $g$ is a convex function on $[0,a]$.

Note that in case $a=0$ we necessarily have $K\subset \{(x,y)~:~
x\leq 0\}$, and so the degenerate case becomes trivial as regards
proving $K\cap \{(x,y)~:~ x\geq 0\} \subset F_k \cap \{(x,y)~:~
x\geq 0\}$. Therefore, we can assume that we have the
non-degenerate case.

Similarly to \eqref{FAmdef}, we define $A_m=(a_m,b_m)$ for
$m=0,\dots,k$ (here $A_0:=(0,0)=O)$, and consider the function
\begin{equation}\label{secondfdef}
f(x):=\begin{cases} \dots \\ (x-a_m)\tan \frac{m\pi}{2k} +a_{m}
\qquad (a_m\leq x \leq a_{m+1}) \\
\dots
\end{cases} \qquad (m=0,\dots,k-1).
\end{equation}
Moreover, now $L$ will denote the broken line joining
$O=A_0,A_1,\dots,A_{k}, \frac12(A_k+A_{k+1})=(a_k,a_k)$ in this
order, that is,
\begin{equation}\label{LFdef}
L:= \{(x,f(x))~:~ 0\leq x \leq a_k \}\cup [A_k,(a_k,a_k)].
\end{equation}
We write
\begin{equation}\label{L1def}
L_1:= L \cup \{(x,0)~:~ x \leq 0\} \cup \{(a_k,y)~:~ y \geq a_k \}
.
\end{equation}
Then $\RR^2\setminus L_1$ will have two connected components; the
convex one will be denoted by $K_1$.

\begin{lemma}\label{lemma:Omega} Let $O\in K=\partial \gamma$,
$(0,-1)$ be an outer normal to $K$ at $O$, and $\Omega(\bfn,1)\geq
\varphi=\frac{\pi}{2k}$. With the notations above, we have
\begin{enumerate}
\item[\rm (i)] $a\leq a_k=\frac12 (\cot\frac{\pi}{4k} +1)(=
{\mathfrak r}(\varphi))$.
\item[\rm (ii)] $0\leq f(x) \leq g(x)$ for all for all $x\in
[0,a]$.
\item[\rm (iii)] $f'_{\pm}(x) \leq g'_{\pm}(x)$ $x\in (0,a)$ and
$f'_{+}(0)\leq g'_{+}(0)$, $f'_{-}(a)\leq g'_{-}(a)$.
\item[\rm (iv)] $b:=g(a)\leq a_{k}$.
\item[\rm (v)] $K \subset K_1$.
\end{enumerate}
\end{lemma}
\begin{proof} Since the degenerate case $a=0$ is trivial (observe
that (ii) is then undefinied, but cf. the paragraph before
\eqref{secondfdef} ), we assume $a>0$.

Let $a^{*}:=\min(a,a_k)$. We argue by induction on $m$, where
$m=0,\dots k-1$, and the inductive assertions will comprise
\begin{enumerate}
\item[\rm (i')] Either $a \leq a_m$ or both(ii') and (iii')
hold, where
\item[\rm (ii')] $0\leq f(x) \leq g(x)$ for all $x\in
[a_m,\min(a,a_{m+1})]$;
\item[\rm (iii')] $g'_{\pm}(x) \geq f'_{\pm}(x)$ for all
$x\in [a_m,\min(a,a_{m+1})]$ (except for $g'_{-}(0)$ and also for
$g'_{+}(a)$ if the second occurs).
\end{enumerate}
Clearly, if we show this for all $m=0,\dots,k-1$ then (i)-(iii) of
the Lemma will be proved.

Let us start with $m=0$. Since $O=(0,0)\in\gamma$, we have
$g(a_{m})=g(0)=0\geq f(a_{m})=f(0)=0$. Let $S:=\{x\in [0,a_1] ~:~
g|_{[0,x]} \geq f|_{[0,x]} \}$. Clearly, by continuity of $f$ and
$g$, $S$ is a closed interval with left endpoint 0. Our aim is to
prove that $S=[0,\min(a,1)]$. Indeed, since $f|_{[0,1]}\equiv 0$,
and as $(0,-1)$ is normal to $K$ at $O$, we must have $g(x)\geq 0$
for all $0\leq x \leq a$, as stated in (ii'). Moreover, since $g$
is a convex curve, $g'_{\pm}(x) \geq 0=f'_{\pm}(x)$ for all $x\in
(0,\min(a,1))$, and also $g'_{+}(0)\geq f'_{+}(0)=0$, furthermore,
$g'_{-}(\min(a,1))\geq f'_{-}(\min(a,1))$. It remains to show
$g'_{+}(1)\geq \tan \varphi = f'_{+}(1)$ in case $\min(a,1)=1$.
But in this case either $a=1$, and then $g'_{+}(1)$ does not exist
(and the case is listed as exceptional in (iii')), or in view of
$|O-(1,g(1))|\geq 1$ any point $(x,g(x))$ along $\gamma$ in the
counterclockwise sense after $(1,g(1))$ but before $P$ (that is,
with $1<x<a$) is of distance $>1$ from $O$, hence by condition its
any outer normal direction is at least $\varphi$ larger than that
of the outer normal $(0,-1)$ of $O$: it follows that $g'_{\pm}(x)
\geq \tan \varphi$ and thus $g'_{+}(1) \geq \tan \varphi=
f'_{+}(1)$.

We proceed by induction. Let $1\leq m<k$ and assume the assertion
for all $0\leq m'<m$. If $\min(a,a_m)=a$, then (i') holds and we
have nothing to prove. Let now $a^{*}_{m+1}:= \min(a,a_{m+1})$. If
$\min(a,a_m)=a_m<a$, then by the inductive assumption we must have
$g(a_m)\geq f(a_m)$ and $g'_{-}(a_m)\geq f'_{-}(a_m)$,
$g'_{+}(a_m)\geq f'_{+}(a_m)=\tan(m\varphi) \equiv
f'_{\pm}|_{(a_m,a_{m+1})}$. In view of convexity we thus obtain
$g'_{\pm}|_{(a_m,a^{*}_{m+1})}\geq g'_{+}(a_m)\geq
f'_{+}(a_m)=\tan(m\varphi) \equiv f'_{\pm}|_{(a_m,a^{*}_{m+1})}$
and by left continuity of the left hand derivative this extends to
$g'_{-}(a^{*}_{m+1})\geq f'_{-}(a^{*}_{m+1})$, too.
%%%%%%%%%%%%%%%%%%%%%%%%%%%%%%%%%%%%%%%%%%%%%%%%%%%%%%%%%%%%%
\comment{Moreover, we must also have $g'_{+}(a^{*}_{m+1})\geq
f'_{+}(a^{*}_{m+1})$ if $a^{*}_{m+1}<a_{m+1}$, since then
$f'_{+}(a^{*}_{m+1})=\tan(m\varphi) = f'_{-}(a^{*}_{m+1})$, as
before.} %%%%%%%%%%%%%%%%%%%%%%%%%%%%%%%%%%%%%%%%%%%%%%%%%%%
Furthermore, if $a^{*}_{m+1}=a_{m+1}$, i.e. $a<a_{m+1}$, then
$g'_{+}(a^{*}_{m+1})=g_{+}(a)$ does not exist (and is listed in
(iii') as exceptional). %%%%%%%%%%%%%%%%%%%%%%%%%%%%%%%%%%%%%%%
The only case remaining is when $a^{*}_{m+1}=a_{m+1}$, i.e. $a\geq
a_{m+1}$. Let first $a=a_{m+1}$. As before, in this case
$g'_{+}(a^{*}_{m+1})=g'_{+}(a)$ does not exist and is excepted in
(iii'). Let now $a>a_{m+1}$, and consider a small right
neighborhood $[a_{m+1}, a_{m+1}+\epsilon]$ of $a_{m+1}$ which is
contained fully in $[0,a)$. Then in this neighborhood the
parametrization $\gamma(x)=(x,g(x))$ extends for a small arc of
$\gamma$ in the counterclockwise sense from
$P_{m+1}:=(a_{m+1},g(a_{m+1}))$, hence for this arc the condition
on $\Omega$ can be applied. (We will use also the notations
$P_0,P_1,\dots,P_m$ defined analogously as $P_k:=(a_k,g(a_k))$,
$k=0,1,\dots,m$).

First we prove that $|P_m-P_{m+1}|\geq 1$, which will also imply
$|P_m-(x,g(x))|> 1$ for all $x\in [a_{m+1}, a_{m+1}+\epsilon]$,
too. For this purpose consider the line
$\ell(x):=P_m+(\tan(m\varphi)) (x-a_m)$ and let
$Q:=Q_m:=(a_{m+1},\ell(a_{m+1}))$. Note that between $a_m$ and
$a_{m+1}$ the line $\ell$ runs below the curve of $\gamma$, since
for any point $x$ between the endpoints $g'_{\pm}(x) \geq
f'_{\pm}(x)=\tan(m\varphi) = \ell'(x)$, and $g(a_m)=\ell(a_m)$. It
follows that $g(a_{m+1})\geq \ell(a_{m+1})$ and thus
$|P_{m+1}-P_{m}|^2\geq (a_{m+1}-a_m)^2+(\tan(m\varphi) \cdot
(a_{m+1}-a_m))^2 = 1$, as stated.

Hence $|P_m-(x,g(x))|> 1$ for all $x\in [a_{m+1},
a_{m+1}+\epsilon]$ holds and the $\Omega$-condition can be applied
to get $\arctan g'_{\pm}(x)\geq g'_{+}(a_m) +\varphi \geq
f'_{+}(a_m)+\varphi =(m+1)\varphi= \arctan f'_{\pm}(x)$. In view
of the right continuity of the right hand derivative, we thus
obtain $g'_{+}(a_{m+1}) \geq \tan ((m+1)\varphi) =
f'_{+}(a_{m+1})$, too.

Therefore, in case (i') does not hold, we conclude (iii'). Since
in this case we have $f(a_m)\leq g(a_m)$ by the inductive
hypothesis, a simple integration using (iii') proves also (ii').

Therefore, the inductive argument for (i')-(iii') concludes and we
obtain (i)-(iii) of the Lemma. It remains to show (iv) and (v). To
prove (iv), it suffices to consider the curve
$\widehat{\gamma}:=T(U_{-\pi/2}(\gamma-(a,b))$ $\gamma_1$ from the
proof of Lemma \ref{lemma:omega}, which will have
$\widehat{P}=(b,a)$ while satisfying all our requirements.

%%%%%%%%%%%%%%%%%%%%%%%%%%%%%%%%%%%%%%%%%%%%%%%%%%%%%%%%%%%%%%%

Finally, let us prove (v): clearly it suffices to prove $\intt
K\subset K_1$. Because at $O$ $K$ has an outer normal $(0,-1)$, we
have $\intt K\subset \{(x,y)~:~y>0\}$. Similarly, as at $P=(a,b)$
$K$ has an outer normal $(1,0)$, in view of Lemma
\ref{lemma:Omega} (i) we also have $\intt K\subset
\{(x,y)~:~x<a\}\subset \{(x,y)~:~x<a_k\}$.

In view of $\intt K\subset \{(x,y)~:~x<a\}$, it remains to show
that $(x,y)\in \intt K$, $0<x<a$ imply $y>f(x)$. However, the part
of $\partial K$ above the open segment $(O,(a,0))$ consists of two
open arcs, the lower one being $\{(x,g(x))~:~0<x<a\}$. Thus, for
$0<x<a$, $(x,y)\in\intt K$ we necessarily have $y>g(x)\geq f(x)$,
as was to be shown.
\end{proof}

\begin{lemma}\label{lemma:K1plus} Let $K,L,L_1,K_1$ as above. Let
$L_1+(u,v)$ a translate of $L_1$ such that $\intt K \subset
K_1+(u,v)$. Further, let $u'\geq u$ and $v'\leq v$. Then also
$\intt K\subset K_1+(u',v')$ holds.
\end{lemma}
\begin{proof} In fact, we are to prove that $K_1\subset
K_1+(w,z)$, with arbitrary $w \geq 0\geq z$. (Then this can be
applied with $(w,z)=(u'-u,v'-v)$ to get $K_1+(u,v)\subset
(K_1+(w,z))+(u,v)=K_1+(u',v')$.) Observe that the special cases
with one coordinate of the translation vector being zero already
suffice, for $K_1\subset K_1+(w,0)\subset
(K_1+(0,z))+(0,w)=K_1+(w,z)$ gives the general case, too. Also
observe that by symmetry of $K_1$ to the line $y=-x$, it suffices
to prove one such case, e.g. $K_1\subset K_1+(0,z)$. However, as
$K_1$ can be defined as the set of points above a function graph,
this last inclusion with$z\leq 0$ is evident.
\end{proof}

\emph{Continuation of the proof of Theorem \ref{th:outscribed}.}
Recall that $T$ is the reflection on the $y$-axis; let us
introduce also $S$ as the reflection on the line $y=a_k$.

From the above argument -- or just reflecting $L$ to the $y$-axis
-- it is immediate that we have also $K\subset TK_1$.
%%%% even the broken line $\widetilde{L}$
%%%%% joining $\frac12(A_{3k}+A_{3k+1}), A_{3k+1}, \dots, A_{4k-1},
%%%% A_{4k}=(-1,0), A_0=O$ along the boundary of $F_k$ does not
%%%% intersect $\intt K$, too.

We are left with the upper part joining $\frac12(A_{k}+A_{k+1}),
A_{k+1}, \dots, A_{3k}, \frac12(A_{3k}+A_{3k+1})$. Let
\begin{align}\label{Lpmdefsecond}
L^{+}:&=\left[\frac{A_{k}+A_{k+1}}{2}, A_{k+1}\right] \cup
\bigcup_{m=k+1}^{2k-1} [A_{m},A_{m+1}] \cup
\left[A_{2k},\frac{A_{2k}+A_{2k+1}}{2}\right], \\
L^{-}:&=\left[\frac{A_{2k}+A_{2k+1}}{2}, A_{2k+1}\right] \cup
\bigcup_{m=2k+1}^{3k-1} [A_{m},A_{m+1}] \cup
\left[A_{3k},\frac{A_{3k}+A_{3k+1}}{2}\right].
\end{align}

Next, let us apply Lemma \ref{lemma:Omega} to the curve from
$P=(a,b)$ onwards to the counterclockwise sense. That is, take
$K_{+}:=U_{-\pi/2}(K-P)$ and $\gamma_{+}:=U_{-\pi/2}(\gamma-P)$
and check that $O\in \gamma_{+}$ and also $\gamma_{+}$ has normal
$(0,-1)$ at $O$; moreover, the same estimate on the minimal
oscillation of the normal holds for $\gamma_{+}$. Thus we obtain
that $K\subset S K_1+(a-a_k,b-a_k) \subset SK_1$, where the last
inclusion follows from $a,b\leq a_k$ and Lemma \ref{lemma:K1plus}.

%%%%%%%%%%%%%%%%%%%%%%%%%%%%%%%%%%%%%%%%%%%%%%%%%%%%%%%%%%%%%%%%
\comment{ $\cap \intt K_{+}=\emptyset $, that is, $U_{\pi/2}(L)+P
\cap \intt K =\emptyset $. Observe $U_{\pi/2}(L)=L^{+}-(a_k,a_k)$,
which entails $L^{+}+(a-a_k,b-a_k)\cap \intt K =\emptyset $. It
suffices to say that $L^{+}+(u,v) \cap \intt K =\emptyset $ with
$u,v\leq 0$. %%%%%%%%%%%%%%%%%%%%%%%%%%%%%%%%%%%%%%%%%%%%%%%%%
} %%%%%%%%%%%%%%%%%%%%%%%%%%%%%%%%%%%%%%%%%%%%%%%%%%%%%%%%%%%%%

Very similarly (or from this and using reflection) we also obtain
%%%%%%%%%%%%%%%%%%%%%%%%%%%%%%%%%%%%%%%%%%%%%%%%%%%%%%%%%%%%%%%%
\comment{ $L^{-}+(p,q)\cap \intt K =\emptyset$ with $p\geq 0$ and
$q\leq 0$. Here $(p,q)=(\widetilde{a}+a_k,\widetilde{b}-a_k)$ with
$\widetilde{P}=(\widetilde{a},\widetilde{b})$ being the first
point of $\gamma$ (from $O$ towards the clockwise direction) with
$\bfn(\widetilde{P})=(-1,0)$. %%%%%%%%%%%%%%%%%%%%%%%%%%%%%%%%%%
}  %%%%%%%%%%%%%%%%%%%%%%%%%%%%%%%%%%%%%%%%%%%%%%%%%%%%%%%%%%%%%
$K\subset TSK_1$. So putting together the four inclusions, we
obtain $K\subset K_1\cap TK_1\cap SK_1\cap TSK_1=F_k$, i.e.
$K\subset F_k$, and the proof concludes.

%%%%%%%%%%%%%%%%%%%%%%%%%%%%%%%%%%%%%%%%%%%%%%%%%%%%%%%%%%%%%%%%%
\comment{  Finally, let us use the segments $[(a_k,a_k),P]$ and
$[(-a_k,a_k),\widetilde{P}]$ to join the curves $\widetilde{L}$,
$L^{+}+(u,v)$ and $L^{-}+(p,q)$, respectively. Clearly, this way
we get a closed, simple Jordan curve $L^{*}$ (a closed, simple
broken line) which is contained in $F_k$ (which is the convex body
enclosed by $L$). Consider the domain $K^{*}$ inside the curve
$L^{*}$. Note that $(0,\epsilon)\in K^{*}\cap \intt K$ (for
$\epsilon$ small enough), hence $K^{*}\cap \intt K\ne \emptyset$.
However, we claim $L^{*} \cap \intt K =\emptyset$. Indeed, we have
already shown this for the parts $\widetilde{L}$, $L^{+}+(u,v)$
and $L^{-}+(p,q)$, and the segment $[(a_k,a_k),P]$ is fully to the
right, while the segment $[(-a_k,a_k),\widetilde{P}]$ is fully to
the left from $\intt K$ in view of the horizontal normals at $P$
and $\widetilde{P}$, respectively. Therefore, $\intt K \cap
L^{*}=\emptyset$, and it follows that $\intt K \subset \intt K^{*}
\subset F_k$, hence $K\subset F_k$, and the proof concludes. %%%%%%%%
}   %%%%%%%%%%%%%%%%%%%%%%%%%%%%%%%%%%%%%%%%%%%%%%%%%%%%%%%%%%%%%%%%
\end{proof}

\section{Further consequences}\label{sec:further}

As the first corollaries, we can immediately deduce the classical
Blaschke theorems. We denote by $D(\x,r)$ the closed disc of
centre $\x$ and radius $r$.

\begin{proof}[Proof of Theorem
\ref{oldth:Blaschke}.] Let $\tau_0$ be the bound provided by (i)
of Proposition \ref{prop:bounds}. Under the condition, we find
(with $\omega(\bfn,\tau)<\pi/2$)
\begin{equation}\label{omegaestimate}
\omega(\bfn,\tau)\leq \frac{\kappa_0 \tau}
{\cos(\omega(\bfn,\tau))}=:\varphi(\tau) \qquad (\tau\leq \tau_0).
\end{equation}
Let us apply Theorem \ref{th:inscribed} for the boundary point
$\x\in\gamma$ with normal vector $\nx=(\sin\alpha,-\cos\alpha)$.
If necessary, we have to reduce $\tau$ so that the hypothesis
$\varphi(\tau)\leq \pi/4$ should hold. We obtain that the
congruent copy $U_{\alpha}(\tau M_k) + \x$ of $\tau M_k$ is
contained in $K$, where $k=\Floor{\pi/2\varphi(\tau)}$. Note that
$U_{\alpha}(\tau M_k) + \x\supset D(\z,\tau r(\varphi(\tau)))$,
where $\z=\x-\tau R(\varphi(\tau))\nx$. When $\tau\to 0$, also
$\varphi(\tau)\to 0$, therefore also $\omega(\bfn,\tau)\to 0$ in
view of \eqref{omegaestimate}, and we see
$$
\lim_{\tau\to 0} \left(\tau R(\varphi(\tau))\right) =
\lim_{\tau\to 0} \left( \tau r(\varphi(\tau))\right)=
\lim_{\tau\to 0} \frac{\tau}{\varphi(\tau)} = \lim_{\tau\to 0}
\frac{\cos(\omega(\bfn,\tau))}{\kappa_0}= \frac{1}{\kappa_0}.
$$
Note that we have made use of $\omega(\bfn,\tau)\to 0$ in the form
$\cos(\omega(\bfn,\tau))\to 1$. It follows that
$D(\x-\frac{1}{\kappa_0}\nx,\frac{1}{\kappa_0})\subset K$, whence
the assertion.
\end{proof}

Note that in the above proof of Theorem A we did not assume
$C^2$-boundary, as is usual, but only the existence of curvature
and the estimate $\kappa(\x) \leq \kappa_0$. So we found the
following stronger corollary (still surely
well-known).%%%% \rev{Reference?}

\begin{corollary}\label{corollary:Astrong} Assume
that $K\subset \RR^2$ is a convex domain with boundary curve
$\gamma$, that the curvature $\kappa$ exists all over $\gamma$,
and that there exists a positive constant $\kappa_0>0$ so that
$\kappa\leq \kappa_0$ everywhere on $\gamma$. Then to all boundary
point $\x\in\gamma$ there exists a disk $D_R$ of radius
$R=1/\kappa_0$, such that $\x\in\partial D_R$, and $D_R\subset K$.
\end{corollary}

Similarly, one can deduce also the "dual" Blaschke theorem, i.e.
Theorem \ref{thold:Blaschkeout}, in a similarly strengthened form.
In fact, the conditions can be relaxed even further, as was shown
by Strantzen, see \cite[Lemma 9.11]{BS}. Our discrete approach
easily implies Strantzen's strengthened version, originally
obtained along different lines.

\begin{corollary}[\bf Strantzen]\label{th:Strantzen} Let $K\subset \RR^2$
be a convex body with boundary curve $\gamma$. Assume that the
(linearly) a.e. existing curvature $\kappa$ of $\gamma$ satisfies
$\kappa\geq \kappa_0$ (linearly) a.e. on $\gamma$. Then to all
boundary point $\x\in\gamma$ there exists a disk $D_R$ of radius
$R=1/\kappa_0$, such that $\x\in\partial D_R$, and $K\subset D_R$.
\end{corollary}

\begin{proof} Now we start with (ii) of Proposition
\ref{prop:bounds} to obtain  $\Omega(\tau)\geq \kappa_0\tau$ for
all $\tau$. Put $\varphi:=\varphi(\tau):=\kappa_0\tau$. Clearly,
when $\tau\to 0$, then also $\varphi(\tau)\to 0$ and $k:=\lceil
\pi/(2\varphi(\tau))\rceil \to \infty$. Take
$\nx=(\cos\alpha,\sin\alpha)$ and apply Theorem
\ref{th:outscribed} to obtain $U_{\alpha}(\tau F_k)+\x\supset K$
for all $\tau>0$. Observe that $D_{\varphi}:=D((0,{\mathfrak
r}(\varphi)),{\mathfrak R}(\varphi)) \supset F_k$, hence
$U_{\alpha}(\tau D_{\varphi})+\x\supset K$. In the limit, since
${\mathfrak r}(\varphi(\tau)) \sim {\mathfrak R}(\varphi(\tau))
\sim 1/(\varphi(\tau)) =1/(\kappa_0\tau)$, we find
$D(\x-(1/\kappa_0)\bfn, 1/\kappa_0) \supset K$, for any $\bfn\in
\nx$, that implies the statement.
\end{proof}

\section{The case of higher dimensional spaces}\label{s:highdim}

The Blaschke and Strantzen theorems in $\RR^d$, $d\geq 2$ can
easily be deduced from the $\RR^2$ versions. However, it is more
difficult to establish $d$-variable analogs of the above discrete
Blaschke type theorems. The main reason for this difficulty is
that normal vectors need not vary within one plane, when $\x$
varies along a plane curve on $\partial K$.

In particular, we do not see if it can happen that for some convex
body $K\subset \RR^d$ the minimal oscillation function
$\Omega(\bfn,\tau)$ is relatively large, while on the restriction
to some plane $P$ -- that is, for $K_0:=K\cap P$ -- it is almost
zero. It can not happen with exactly zero, as then a straight line
segment $L$ belongs to $\partial K_0$, and the space normals at
relative interior points of $L$ define supporting hyperplanes in
the space which are valid also for all other relative interior
points of $L$, therefore, $\nx=\ny$ for all points $\x,\y \in {\rm
relint} L$. But if with a circle arc $C$ of large radius (so of
small curvature) we assume $C\subset \partial K_0$, and we define
a twisting function $\nx$ along $C$, then the halfspaces with
outer normals $\nx$ through $\x$ mesh in a convex set $K$ with
nonempty interior: it is then not too difficult to make K into a
convex body with a few other halfspaces. So at least along the arc
$C$ the space normals change considerably, while not in $P$.
However, it is not clear if a construction can be made with
$\Omega(\bfn,\tau)$ large and $\Omega(\bnu,\tau)$ small when
considered globally on $K$ and $K_0$, resp. (What is missing in
the above idea is to guarantee that the space normals do change
considerably \emph{between any two points} of distance $\geq
\tau$, and not only between $\tau$-far points of $C$.)

In any case, we present a version, even if somewhat weaker than
one would wish, involving the modulus of continuity with respect
to chord length. Again, it is unclear if an extension with respect
to geodesic distance on $\partial K$, i.e. the natural extension
of arc length in $\RR^2$, can be established. On the other hand,
we will be able to formulate our statements in the generality of
infinite dimensional spaces. In fact, here it will be more
convenient to consider another form of the modulus of continuity,
easily defined even in Banach spaces, where distance of the outer
unit normals will be measured in the distance of vectors within
the dual space (i.e. chord length), and not the geodesic distance
(i.e. angle difference) on the surface of the unit ball of the
dual space. So we can introduce the next definitions.

\begin{definition}\label{def:spaceomega}
Let $\bfv:H\to 2^M\setminus\{\emptyset\}$, where $H\subset X$,
$M\subset Y$ are sets in the normed vector spaces $X,Y$,
respectively. Define
\begin{equation}\label{modcontspace}
\overline{\omega}(\tau):=\overline{\omega}(\bfv,\tau):=\sup \{
\|\bfu -\bfw\|_Y \,:~\,\x,\y\in H,~\|\x-\y\|_X \leq \tau,~ \bfu\in
\bfv(\x),\bfw\in\bfv(\y) \}.
\end{equation}
and
\begin{equation}\label{minoscspace}
\overline{\Omega}(\tau):=\overline{\Omega}(\bfv,\tau):=\inf \{
\|\bfu-\bfw\|_Y ~:~ \x,\y\in H,~\|\x-\y\|_X \geq \tau ,~ \bfu\in
\bfv(\x),~\bfw\in\bfv(\y)\}.
\end{equation}
\end{definition}

In our use of the notion, we will take for $H$ either $\partial K$
or $\partial K\cap P$, equipped with the norm distance from $X$,
and $M$ will be the unit ball in the dual space $X^{*}$. In fact,
due to the geometrical nature of our subject, we will also need
orthogonality, that is, Hilbert space structure, in the estimation
of Proposition \ref{prop:modulicompare} below.

The main reason to use this type of distance in measuring the
change of the outer unit normal vectors is that angles can not be
handled elegantly when an angle is multiplied by some scalar (like
$2R/r$ below), for the arising arcsin etc. functions have
restricted domain. On the other hand, increase of distance
estimates are naturally and easily formulated.

Recall that any {\it{infinite dimensional real Hausdorff
topological linear space}} we mean by a {\it{convex body}} a
bounded closed convex set with non-empty interior. Let $K$ be a
convex body in an infinite dimensional real Hilbert space $X$. We
have $K=\overline{{{\text{int}}K}}$ (this holding in any Hausdorff
topological vector space, \cite[I, p. 413, Theorem 1, (c)]{DS}. We
say that $K$ is a {\it{smooth convex body}} if for any ${\bold x}
\in {\partial}K$ there exists exactly one unit vector ${\bold
n}({\bold x})$, called {\it{outer unit normal of $K$ at ${\bold
x}$}}, such that $K \subset \{ {\bold y} \in X \mid \langle {\bold
y}, {\bold n}({\bold x}) \rangle \le \langle {\bold x}, {\bold
n}({\bold x}) \rangle \} $. Observe that at least one such vector
exists, since ${\bold x} \notin {\text{int}}K \ne \emptyset $ (and
in a Hausdorff topological vector space a non-empty open convex
set and a convex set disjoint to it can be separated by a non-zero
continuous linear functional, cf. \cite[I, p. 417, Theorem
8]{DS}.)

For $\c\in\RR^d$ and $r>0$ we let $B(\c,r)$ the closed ball of
centre $\c$ and radius $r$.

\begin{proposition}\label{prop:modulicompare} Let $K\subset H$
be a convex body in the Hilbert space $H$. Assume that $\c\in K$
with $B(\c,r)\subset K \subset B(\c,R)$, where $0<r<R$. Take any
two-dimensional plane $P$ through $\c$, denote $K_0:=K\cap P
\subset P$, and denote $\bnu(\x)$ the (in general multivalued)
outer unit normal vector function of $K_0$ at $\x\in\partial K_0$
within $P$.

Let $\overline{\omega}(\n,\tau)$ and
$\overline{\omega}(\bnu,\tau)$ stand for the modulus of continuity
\eqref{modcontspace} of the surface normal vectors $\n$ of $K$
along $\partial K$ in $H$ and $\bnu$ of $K_0$ along $\partial K_0$
in $P\subset H$, respectively. We then have
\begin{equation}\label{modcontcompare}
\overline{\omega}(\bnu,\tau)\leq \frac{2R}{r}
\overline{\omega}(\bfn,\tau).
\end{equation}
\end{proposition}

The following lemmas are well-known elementary facts of space
geometry.

\begin{lemma}\label{lemma:nonzeroproj} Let $K \subset H$ be a
convex body in the Hilbert space $H$, let $\x \in \partial K$ and
$P$ be any affine subspace through the point $\x$ and containing
some interior point of $K$, too (so that, in particular, $P$ has
dimension at least 1). Assume that $P$ is closed (which is
satisfied in any case if $\dim P<\infty$) and denote $\Pi:=\Pi_P$
the orthogonal projection of vectors to the affine subspace $P$.
Then for any $\bfu\in \nx$ we have $\Pi \bfu \ne {\bf 0}$.
\end{lemma}
\begin{proof}
Let $\bfu \in \nx$ be an outer unit normal vector to $K$ at $\x$.
In order to prove $\Pi \bfu \ne {\bf 0}$, it suffices to take any
$\y\in \intt K\cap P$ (which exists by assumption), and show that
the vector $\y-\x$ is not orthogonal to $\bfu$. Indeed, for a
small ball $B=B({\bf 0},\delta)$ such that $\y+B\subset \intt K $,
$\< \bfu, \y+\bbb - \x\> \leq 0$ for all $\bbb \in B$ (because
$\bfu\in \nx$) and $\<\bfu , \y-\x\> = 0$ would imply $\< \bfu,
\bbb \> \leq 0$ for the whole ball $B$, hence $\< \bfu, \bbb \>
=0$ ($\forall \bbb \in B$) and $\bfu={\bf 0}$, a contradiction.
\end{proof}

\begin{lemma}\label{lemma:projections} Let $K \subset H$ be a
convex body in a Hilbert space $H$, $\x \in \partial K$ and $P$ be
any closed affine subspace through the point $\x$ and containing
some interior point of $K$, too. Denote $\Pi:=\Pi_P$ the
orthogonal projection of vectors to $P$ and write $\nux$ for the
nonempty set of all outer unit normals to $K_0$ in $P$ at $\x$.

Then for arbitrary $\w \in \nux$ there exists some $ \bfu \in \nx$
such that ${\bf 0}\ne \Pi \bfu \| \w$. Conversely, if $\bfu\in
\nx$ then $\Pi \bfu \ne {\bf 0}$ and the unit vector $\w:=\Pi
\bfu/|\Pi \bfu|$ belongs to $\nux$.
\end{lemma}
\begin{proof}
Let now $\w\in \bnu(\x)$. Consider the orthogonal complement $V$
of $\w$ in $P-\x$, i.e. $V:=\{ \bfv\in P-\x~:~ \<\w,\bfv\> =0\}$,
take $V_0:=V\cap B({\bf 0},1)$, and consider the new convex set
$M$ generated as the convex combination of $V_0+\x (\subset P)$
and $K$, i.e. $M:={\{ t \z +(1-t) \y~:~ \z \in K,~\y \in V_0+\x,~0
\leq t \leq 1\}}$. Note that ${\rm cl M}$ is a convex body.
Clearly $M_0:=M\cap P = \{ t \z +(1-t) \y~:~ \z\in K_0,~\y\in
V_0+\x,~0 \leq t \leq 1\}$. Observe that $\w$ is still an outer
unit normal vector in $P$ at $\x$ even to $M_0$.

Consider the disjoint nonempty convex sets $\intt M-\x$ and $[{\bf
0},\w]$. Note that $\intt M\supset \intt K\ne \emptyset$ is open.
It follows that there exists a normalized linear functional,
whence a unit vector $\bfu$, such that $\langle \bfu, \y-\x
\rangle < 0 \leq \langle \bfu, \w \rangle $ for all $\y \in \intt
M$. Obviously by convexity of $M$ and in view of $\intt M\ne
\emptyset$, we have $M \subset {\rm cl}\intt M$. Therefore we find
$\langle \bfu, \y-\x \rangle \leq 0 \leq \langle \bfu, \w \rangle
$ for all $\y\in M$, hence also for vectors in $K$. So clearly any
such $\bfu$ belongs to $\nx$. Note that by Lemma
\ref{lemma:nonzeroproj} $\Pi \bfu \ne {\bf 0}$.

Moreover, $\<\bfu, (\x+\bfv)-\x\>\leq 0$ for all $\bfv\in V_0$, so
that $\bfu \bot V$. Therefore, $\Pi \bfu \bot V$ and thus $\Pi
\bfu \| \w$, as needed.

Finally, consider the converse: let $\bfu\in \nx$ (so that $\Pi
\bfu \ne {\bf 0}$) and write $\bfu = \Pi \bfu + \bfu'$, where
$\bfu' \bot P$ (i.e., $\bfu \bot (P-\x)$). Clearly, $\< \bfu' ,
\y-\x\> =0$ for all $\y\in K_0\subset P$. On the other hand
$\<\bfu,\y-\x\>\leq 0$ for all $\y\in K$ (because $\bfu \in \nx$),
hence for all $\y\in K_0$, so combining these two we get $\<\Pi
\bfu,\y-\x\> \leq 0$ for all $\y\in K_0$. So we can take $\w:=\Pi
\bfu / |\Pi \bfu |$, which satisfies $\<\w,\y-\x\> \leq 0$ for all
$\y\in K_0$, whence $\w\in \bnu(\x)$ and the assertion follows.
\end{proof}

\begin{lemma}\label{lemma:spacechange} Let $K
 \subset H$ be a convex body in the Hilbert space $H$, and
$\bfc \in \intt K$ satisfying $B(\c,r)\subset K \subset B(\c,R)$
with some $0<r<R<\infty$. Let $\x,\y\in \partial K$. Denote
$K_0:=K\cap P$, where $P$ is a closed affine subspace containing
the points $\x, \y$, and $\c$. (Note that $P$ is at least 1
dimensional.) Let us write $\Pi_P:=\Pi$ for the orthogonal
projection of vectors to $P$.

Denote the nonempty set of all outer unit normals to $K_0$ in $P$
at $\x$ and $\y$ as $\nux$ and $\nuy$, respectively, and let
$\w\in \nux$, $\z\in\nuy$ be arbitrary. Let $\bfu, \bfv$ be outer
unit normal vectors from $\nx, \ny \subset S_H$, respectively,
such that ${\bf 0} \ne \Pi\bfu \| \w,~{\bf 0} \ne \Pi\bfv\|\z$
(which exist according to Lemma \ref{lemma:projections}). Then we
have
\begin{equation}\label{munueq}
|\w-\z| \leq \frac{2R}{r} |\bfv-\bfu|.
\end{equation}
Consider now the angles $\varphi(\bfu,\bfv):=\arccos \langle
\bfu,\bfv \rangle\in [0,\pi]$ and $\psi(\w,\z):= \arccos \langle
\w,\z \rangle \in [0,\pi]$ between these normals. Moreover, assume
$\varphi(\bfu,\bfv)< 2 \arcsin \frac{r}{2R}$. We then have
\begin{equation}\label{amunueq}
\psi(\w,\z) \leq 2 \arcsin \left(\frac{2R}{r}
\sin\frac{\varphi(\bfu,\bfv)}{2} \right).
\end{equation}
\end{lemma}

\begin{proof} Denote $Z:=Z(\x):=\bfu^{\perp}+\x$. $Z$ being a
supporting affine hyperplane, $\intt K \cap Z =\emptyset$, thus
the projection $\c'$ of $\c$ to $Z$ does not belong to $\intt K$.
That is, $|\c'-\c|\geq r$ in view of $B(\c,r)\subset K$. On the
other hand, $\c,\x \in K$ and $|\x-\c|\leq R$, hence reading from
the right angle triangle of $\c'\c\x$, the angle $\Phi$ between
the vectors $\x-\c$ and $\c'-\c\parallel \bfu$ is $\Phi = \arccos
(\cos\Phi) \leq\arccos (r/R)$. (In case $\c'=\x$, we also have
$\Phi=0$.) From this we estimate $|\Pi \bfu|$. As $\x,\c\in P$, we
clearly have $|\Pi\bfu| =\cos \angle (P,\bfu) \geq \langle \bfu,
\frac{\x-\c}{|\x-\c|}\rangle =\cos \Phi \geq r/R$.

Observe that by definition $\w=\Pi\bfu/|\Pi\bfu|$ and
$\z=\Pi\bfv/|\Pi\bfv|$. Therefore, we obtain
\begin{align}
|\w-\z|&=\left|\frac{\Pi\bfu}{|\Pi\bfu|} -\frac{\Pi\bfv}
{|\Pi\bfv|}\right|=\left|\frac{\Pi\bfu-\Pi\bfv}{|\Pi\bfu|}
+\frac{|\Pi\bfv|-|\Pi\bfu|}{|\Pi\bfu|}\frac{\Pi\bfv}{|\Pi\bfv|}\right|
\notag \\ & \leq \frac{2|\Pi(\bfv-\bfu)|}{|\Pi\bfu|} \leq
\frac{2|\bfv-\bfu|} {|\Pi\bfu|} \leq \frac{2R}{r} |\bfv-\bfu|,
%%%%%%%%%%%%%%%%%%%%%% \frac{4\sin(\varphi(\bfu,\bfv)/2)}{r/R}~,
\end{align}
proving \eqref{munueq}. Clearly $|\bfu-\bfv|=2
\sin(\varphi(\bfu,\bfv)/2)$, so in case $\varphi(\bfu,\bfv)< 2
\arcsin \frac{r}{2R}$ the right hand side does not exceed 1 and
this yields
\begin{equation}\label{Psivarphi}
\psi(\w,\z)=2 \arcsin \left(\frac{|\w-\z|}{2}\right) \leq 2
\arcsin \left(\frac{2R\sin(\varphi(\bfu,\bfv)/2)}{r}\right),
\end{equation}
and \eqref{amunueq} is proved, too.
\end{proof}

As said above, under some restrictions we can derive certain
estimates between the modulus of continuity \eqref{modcontuv}
(with respect to chord length) of outer unit normal vectors in the
whole space and in a plane section.

\begin{proposition}\label{prop:anglecompare} Let $K\subset H$ be
a convex body, where $H$ is a Hilbert space. Assume that $\c\in K$
with $B(\c,r)\subset K \subset B(\c,R)$ with $0<r<R\leq \infty$.
Take any two-dimensional plane $P$ through $\c$, denote
$K_0:=K\cap P \subset P$, and denote $\bnu(\x)$ the (in general
multivalued) outer unit normal vector function of $K_0$ at
$\x\in\partial K_0$ within $P$. Let $\omega(\n,\tau)$ and
$\omega(\bnu,\tau)$ stand for the modulus of continuity
\eqref{modcontuv} (with respect to chord length) of the surface
normal vectors $\n$ of $K$ along $\partial K\subset H$ and $\bnu$
of $K_0$ along $\partial K_0$ in $P$, respectively. We then have
\begin{equation}\label{modcontcompare}
\omega(\bnu,\tau)\leq 2 \arcsin \left( \frac{2R}{r} \sin
\frac{\omega(\n,\tau)}{2}\right) \quad \textrm{whenever}~~ \tau<
\tau_0,
\end{equation}
where $\tau_0$ is chosen to satisfy $0<\tau_0\leq 2r$ and
$\omega(\n,\tau_0) \leq 2\arcsin \frac{r}{2R}$.
\end{proposition}

\begin{remark} Note that $K\subset B(\c,R)$ holds with $R:={\rm diam}
K-r<{\rm diam} K $, always. On the other hand for a
\emph{symmetric} convex body $K\subset H $ $r=w(K)/2$, where
$w(K)$ is the \emph{minimal width} of $K$, and in general by
Steinhagen's Inequality, $\frac{1}{2\sqrt{d}} w(K) \le r \le
w(K)/2$ (if $d$ is odd) and $\dfrac{\sqrt{d+2}}{2d+2} w(K) \le r
\le w(K)/2$ (if $d$ is even) for $K\subset \RR^d$ any convex body,
see \cite{BF}.
\end{remark}

\begin{proof} Let $P$ be any such plane, and $\x,\y\in\partial
K_0$ satisfying $|\y-\x|\leq \tau <\tau_0$. \comment{Observe that
then $\c$, $\x$ and $\y$ cannot be collinear, unless $\tau=0$ and
$\x=\y\in\partial K$. Indeed, if $\c$ have separated $\x$ and $\y$
on a straight line, then $\tau_0 >\tau= |\x-\y|\geq 2r$ (as
$B(\c,r)\subset K$), contradicting to our assumption $\tau_0\leq
2r$, and if $\c$ have not separated $\x$ and $\y$, but
nevertheless the three points were collinear, then $\x$ and $\y$
would have to sit on the same ray emanating from $\c$, furthermore
we must have $\x=\y$, $\tau=0$, as otherwise one of $\x,\y$ is a
nontrivial convex combination of the other boundary point and
$\c\in \intt K$, hence itself is an interior point, c.f. \cite[I,
p. 413, proof of Theorem 1]{DS}, a contradiction with
$\x,\y\in\partial K$ again.

So consider first the case when $0<\tau<\tau_0$ and $\x\ne\y$.
Then according to the above, $\x,\y,\c$ are not collinear.} By
condition, for any $\bfu\in\nx,~\bfv\in \ny$ we have $\arccos
\langle \bfu,\bfv \rangle \leq \omega(\n,\tau) \leq
\omega(\n,\tau_0) \leq 2\arcsin \frac{r}{2R}$ and Lemma
\ref{lemma:spacechange} applies. Using the vectors $\bfu, \bfv$,
provided by Lemma \ref{lemma:projections} to $\w, \z$, and
applying also $\arccos \langle \bfu,\bfv \rangle \leq
\omega(\n,\tau)$ and Lemma \ref{lemma:spacechange}, we infer for
arbitrary $\w \in \bnu(\x),~\z \in\bnu(\y)$ the estimate
\begin{equation}\label{nuestimate}
\arccos \langle \w,\z \rangle \leq 2 \arcsin
\left(\frac{2R\sin\frac{\arccos \langle
\bfu,\bfv\rangle}{2}}{r}\right) \leq 2 \arcsin
\left(\frac{2R}{2}\sin\frac{\omega(\bfn,\tau)}{r}\right).
\end{equation}
Taking supremum over $\w\in \bnu (\x)$ and $\z\in \bnu(\y)$ we
arrive at \eqref{modcontcompare}, whence the assertion.
\end{proof}

\begin{proof}[Proof of Proposition \ref{prop:modulicompare}] Quite
the same way as above, a direct application of Lemma
\ref{lemma:spacechange}, \eqref{munueq} yields the less restricted
variant formulated above as Proposition \ref{prop:modulicompare}.
\end{proof}

Even if not sharp in the sense that it does not imply the
well-known Blaschke theorem in $\RR^d$ (but only with the loss of
a constant factor, we can still formulate a space version of the
discrete Blaschke type theorem above. Namely, we obtain by an
application of Proposition \ref{prop:anglecompare} and Theorem
\ref{th:inscribed} the following.

\begin{corollary}\label{cor:spacediscrBlaschke} Let $K\subset H$ be
a convex body, where $H$ is a Hilbert space. Assume that $\c\in K$
with $B(\c,r)\subset K \subset B(\c,R)$ with $0<r<R\leq \infty$.
Let us also assume that with certain parameters $0<\tau<2r$ and
$0<\varphi< 2 \arcsin(r/(4\sqrt{2}R))$, the modulus of continuity
\eqref{modcontuv} of the non-empty valued, multivalued outer unit
normal vector function $\nx$ satisfies $\omega (\bfn,\tau)\leq
\varphi$.

Then for any boundary point $\x\in \partial K$, any outer unit
normal vector $\bfu \in \nx$ and any two-dimensional affine plane
$P$ passing through $\x$ and $\c$ the orthogonal projection $\Pi_P
\bfu$ of $\bfu$ to $P$ does not vanish, and
$\w:=\Pi_P\bfu/|\Pi_P\bfu|$ is well-defined. Let us take an
arbitrary coordinate system in $P$ and let us put
$\alpha:=\arg(\w)+\pi/2 \mod 2\pi$ with respect to this coordinate
system. Then with the mangled $n-gon$ $M(\x,\al,\Phi,\tau)$
defined in Definition \ref{def:M}, \eqref{Mdef}, we have
$M(\x,\al,\Phi,\tau) \subset K$, where $\Phi=2 \arcsin \left(
\dfrac{2R}{r} \sin \dfrac{\varphi}{2}\right)$.
\end{corollary}

\begin{proof} Let $\x, \c, P, \bfu, \w, \tau, \varphi, r, R$ as
above. Also let $K_0:=K\cap P$ and $\bnu(\x)$ be the planar outer
unit normal vector(s) function of $K_0$ in $P$. By the final
assertion of Lemma \ref{lemma:projections} we indeed have $\Pi_P
\bfu\ne {\bf 0}$, moreover, $\w\in \bnu(\x)$.

Let us estimate the modulus of continuity $\omega(\bnu, \cdot)$ of
$\bnu(\x)$. By Proposition \ref{prop:anglecompare} we
have $ \omega(\bnu,t)\leq 2 \arcsin \left( \frac{2R}{r} \sin
\frac{\omega(\n,t)}{2}\right)$, whenever $t\leq t_0$ where
$0<t_0<2r$ is such that $\omega(\bfn,t_0)\leq 2 \arcsin
\dfrac{r}{2R}$. Let us choose here $t_0:=\tau$: certainly
$\tau<2r$, and by condition $\omega(\bfn,\tau)=\varphi<2
\arcsin(r/(4\sqrt{2}R)) < 2 \arcsin (r/(2R))$, so the above
inequality from Proposition \ref{prop:anglecompare} applies even
to $\tau$ in place of $t$. We thus find $\omega(\bnu,\tau)\leq 2
\arcsin \left( \frac{2R}{r} \sin
\frac{\omega(\n,\tau)}{2}\right)\leq 2 \arcsin \left( \frac{2R}{r}
\sin \frac{\varphi}{2}\right) < 2 \arcsin \left(
\frac{1}{2\sqrt{2}} \right)<\pi/4 $ as $2 \sigma \leq \arcsin (2
\sin \sigma)$ for $0\leq \sigma \leq \pi/4$.

So now we know that the modulus of continuity of the outer unit
normal vector(s) function $\bnu(\x)$ of $K_0$ in $P$ satisfies
$\omega(\bfn,\tau)\leq \Phi$ with $\Phi:=2 \arcsin \left(
\frac{2R}{r} \sin \frac{\varphi}{2}\right)<\pi/4$.

That is, Theorem \ref{th:inscribed} applies with $\tau$ and $\Phi$
and the outer unit normal $\w\in\bnu(\x)$. Thus we are led to
$M(\x,\al,\Phi,\tau) \subset K$, and the assertion follows.
\end{proof}

\begin{remark} Corollary \ref{cor:spacediscrBlaschke} provides
an inscribed body through description of its plane sections.
However, note that in general the inscribed body is \emph{not} a
rotationally symmetric body, the possible axes of symmetry, $\w\in
\bnu (\x)$, varying from plane section to plane section.
Nevertheless, in the limit, when curvature exists and $\tau$,
whence $\varphi$ tends to 0, one can deduce a Blaschke type
theorem with some ball of radius $\rho=r/(2R\kappa_0)$ inscribed
in $K$ and containing $\x\in \partial K$, as above in Corollary
\ref{corollary:Astrong}.
\end{remark}

\section{Acknowledbement}

This paper was thoroughly checked by Endre Makai, who gave
numerous suggestions for improving the presentation. We thank to
him also the sharpening of the inscribed mangled $n$-gon result,
by means of the modulus of continuity with respect to arc length,
as given above in Theorem \ref{th:arclength}.

\end{document}